\documentclass{ip-journal}
\usepackage{amssymb,amsfonts}
\usepackage[all,arc]{xy}
\usepackage{enumerate}
\usepackage{mathrsfs}
\usepackage{tikz}
\usepackage{tikz-cd}
\usetikzlibrary{arrows,snakes,backgrounds}
\tikzset{>=stealth}
\usepackage{color}   
\usepackage{marginnote}
\usepackage{sseq}
\usepackage{nameref}
\usepackage[backref=page]{hyperref}
\usepackage{MnSymbol} 
\hypersetup{
    colorlinks=true, 
    linktoc=all,     
    citecolor=black,
    filecolor=black,
    linkcolor=blue,
    urlcolor=black
}
\usepackage{cleveref}
\usetikzlibrary{arrows}

\newtheorem{thm}{Theorem}[section]
\newtheorem*{thm*}{Theorem}
\newtheorem{cor}[thm]{Corollary}
\newtheorem{prop}[thm]{Proposition}
\newtheorem{lem}[thm]{Lemma}

\newtheorem{problem}[thm]{Problem}

\theoremstyle{definition}
\newtheorem{defn}[thm]{Definition}

\theoremstyle{remark}
\newtheorem{rem}[thm]{Remark}

\newcommand{\bF}{\mathbb{F}}

\newcommand{\bQ}{\mathbb{Q}}
\newcommand{\bR}{\mathbb{R}}
\newcommand{\bS}{\mathbb{S}}

\newcommand{\bZ}{\mathbb{Z}}

\newcommand\Diff{\mathrm{Diff}}
\newcommand\Ham{\mathrm{Ham}}
\newcommand\eHam{\widetilde{\Ham}}
\newcommand\deHam{\widetilde{\Ham}^{\delta}}
\newcommand\BdeHam{\mathrm{B}\widetilde{\Ham}^{\delta}}
\newcommand\dHam{\mathrm{Ham}^{\delta}}
\newcommand\BdHam{\mathrm{BHam}^{\delta}}
\newcommand\Symp{\mathrm{Symp}}
\newcommand\BSymp{\mathrm{BSymp}}
\newcommand\dSymp{\mathrm{Symp}^{\delta}}
\newcommand\BdSymp{\mathrm{BSymp}^{\delta}}

\newcommand\BH{\mathrm{B}\Gamma_2^{\text{\textnormal{vol}}}}
\newcommand\dDiff{\mathrm{Diff}^{\delta}}

\newcommand\BdDiff{\mathrm{BDiff}^{\delta}}

\newcommand\Coker{\operatorname*{Coker}}
\newcommand\Ker{\operatorname*{Ker}}

\newcommand{\hcoker}{/\!\!/}
\newcommand{\hker}{\backslash\!\!\backslash}

\newcommand{\SnSn}{\#_g S^n \times S^n}

\newcommand{\W}{\text{W}_{g,1}}
\newcommand{\WW}{\text{\textnormal{W}}_{g,1}}

\addtocontents{toc}{\protect\setcounter{tocdepth}{1}}
\makeatletter
\let\c@equation\c@thm
\makeatother
\numberwithin{equation}{section}
\title{On the moduli space of flat symplectic surface bundles}

\author{Sam Nariman}
\email{sam@math.northwestern.edu}
\address{Department of Mathematics\\
  Northwestern University\\
2033 Sheridan Road\\
Evanston, IL  60208}

\begin{document}

\begin{abstract}
In this paper, we prove homological stability of symplectomorphisms and {\it extended} hamiltonians of surfaces made discrete. Similar to discrete surface diffeomorphisms \cite{nariman2015stable}, we construct an isomorphism from the {\it stable} homology group of  symplectomorphisms and {\it extended} Hamiltonians of surfaces to the homology of certain infinite loop spaces.  We use these infinite loop spaces to study characteristic classes of surface bundles whose holonomy groups are area preserving, in particular we give a homotopy theoretic proof of the main theorem in \cite{kotschick2004characteristic}.
\end{abstract}

\maketitle


\section{Introduction and  statement of the main results}
The Madsen-Weiss theorem (\cite{madsen2007stable}) was not only so successful in describing the ``stable" invariants of the surface bundles, but also it laid out a method that could be generalized to higher dimensional manifold bundles (see \cite{galatius2012stable}). Kotschick and Morita in series of papers (see \cite{kotschick2004characteristic}, \cite{kotschick2005signatures} and \cite{kotschick2009gelfand}) studied the invariants of surface bundles whose holonomy groups lie in the symplectomorphisms of surfaces. Their calculations heavily relies on the theory of surfaces. The purpose of this paper which is a continuation of the work in  \cite{nariman2015stable, nariman2014homologicalstability} is to do Madsen-Weiss theory for surface bundles with certain geometric restrictions on the holonomy groups. In this approach the theory of surfaces will be hidden  in the homological stability results (see \Cref{thm1} and \Cref{thm2}). And as we shall see, not only this homotopy theoretic approach in the case of surface bundles whose holonomy groups are area preserving  recovers the main theorems in \cite{kotschick2004characteristic} and \cite{kotschick2005signatures} but also it translates the problems posed in \cite{kotschick2004characteristic} to a concrete problem related to the homotopy type of the Haefliger classifying space of the groupoids of germs of volume preserving diffeomorphisms (see \Cref{defn1}).
\subsection{Homological stability}
Let $\Sigma$ be a surface with or without boundary  and let $\omega_{\Sigma}$ be an area form on  $\Sigma$ whose total volume is normalized to be the negative of the Euler number. Let $\Diff(\Sigma,\partial)$ and $\Symp(\Sigma,\partial)$ denote respectively the group of orientation preserving diffeomorphisms and $\omega_{\Sigma}$-preserving diffeomorphisms of $\Sigma$ whose supports are away from the boundary. We denote the same groups with discrete topology by $\dDiff(\Sigma,\partial)$ and $\dSymp(\Sigma,\partial)$ respectively. The first main theorem
\begin{thm}\label{thm1}
The homology groups $H_*(\dSymp(\Sigma,\partial);\bZ)$ is independent of the genus $g(\Sigma)$ and the number of boundary components  if $*\leq (2g(\Sigma)-2)/3$. 
\end{thm}

Let $\Symp_0(\Sigma,\partial)$ denote the identity component of $\Symp(\Sigma,\partial)$. It is a consequence of a theorem of Moser \cite{MR0182927} that $\Symp_0(\Sigma,\partial)$ is homotopy equivalent to $\Diff_0(\Sigma,\partial)$ which is known (see \cite{earle1969fibre, MR0277000}) to be contractible for $g\geq 2$. Recall that the flux homomorphism $$\text{Flux}: \Symp_0(\Sigma,\partial)\to H^1(\Sigma,\partial;\bR),$$ is a surjective homomorphism that is roughly described as follows. For an element $\phi\in \Symp_0(\Sigma,\partial)$, choose a path $\phi_t$ to the identity. Let   $\alpha$ be $1$-cycle in $\Sigma$, then  $\text{Flux}(\phi)(\alpha)$ is given by integrating $\omega_{\Sigma}$ on the $2$-chain $(s,t)\to \phi_t(\alpha(s))$. In this case, since $\Symp_0(\Sigma,\partial)$ is simply connected, the definition of $\text{Flux}$ does not depend on the path $\phi_t$. The group of Hamiltonians is defined to be the kernel of  $\text{Flux}$, hence they sit in a short exact sequence
\[
1\to \Ham(\Sigma,\partial)\to \Symp_0(\Sigma,\partial)\xrightarrow{\text{Flux}} H^1(\Sigma,\partial;\bR)\to 1.
\]
Morita and Kotschick proved in \cite{kotschick2005signatures} that the flux homomorphism can be extended to a crossed homomorphism
\[
\widetilde{\text{Flux}}: \Symp(\Sigma,\partial)\to H^1(\Sigma,\partial;\bR), 
\]
which is a map that instead of being homomorphism satisfies the identity
\[
\widetilde{\text{Flux}}(fg)=\widetilde{\text{Flux}}(g)+g^*\widetilde{\text{Flux}}(f),
\]
where $g^*$ denotes the action of $g$ on $H^1(\Sigma,\partial;\bR)$. Although $\widetilde{\text{Flux}}$ is not a group homomorphism, its kernel is a subgroup of $\Symp(\Sigma,\partial)$. This kernel is called {\it extended} Hamiltonians and we shall denote it by $\widetilde{\Ham}(\Sigma,\partial)$. The group of extended Hamiltonians is an enlargement of $\Ham(\Sigma,\partial)$ that intersects all the connected components of $\Symp(\Sigma,\partial)$ and sits in a short exact sequence
\[
1\to \Ham(\Sigma,\partial)\to \widetilde{\Ham}(\Sigma,\partial)\to \text{MCG}(\Sigma,\partial)\to 1,
\]
where $\text{MCG}(\Sigma,\partial)$ denotes the mapping class group of the surface $\Sigma$.  Kotschick and Morita in \cite[Theorem 6]{kotschick2004characteristic} proved that the group homology of the Hamiltonians is highly nontrivial and it is not stable with respect to the genus. We prove, however, that the group homology of  $\widetilde{\Ham}^{\delta}(\Sigma,\partial)$ is stable.
\begin{thm}\label{thm2}
Let $\Sigma$ be a surface with at least one boundary component, then the homology groups $H_*(\deHam(\Sigma,\partial);\bZ)$ is independent of the genus $g(\Sigma)$ and the number of boundary components  if $*\leq (2g(\Sigma)-2)/3$. 
\end{thm}
\subsection{The stable homology}To identify the stable homology of $\deHam(\Sigma,\partial)$ and $\dSymp(\Sigma,\partial)$, we first recall the definition of the classifying  space of codimension $2$ foliations with a transverse volume form. 
\begin{defn}\label{defn1}Let $\Gamma_2^{\text{vol}}$ denote the topological Haefliger groupoid whose objects are $\bR^2$ with the usual topology and the space of morphisms are  local symplectomorphisms of $\bR^2$ with respect to the standard symplectic form (see \cite{haefliger1971homotopy} for more details on how this groupoid is topologized). We shall write $\BH$ to denote its classifying space. \end{defn}There is a map 
\[
\theta: \BH\to \mathrm{BSL}_2(\bR),
\]
which is induced by the functor $\Gamma_2^{\text{vol}}\to \mathrm{SL}_2(\bR)$ that  sends a local diffeomorphism to its derivative at its source. We denote the homotopy fiber of $\theta$ by $\overline{\BH}$. Let $v\in H^2(\BH;\bR)$ be the standard transverse volume form for the universal  $\Gamma_2^{\text{vol}}$-structure on $\BH$ (cf. \cite{MR678355}). Let $e\in H^2(\BH;\bR)$ denote the Euler class of the normal bundle of the codimension $2$ Haefliger structure on $\BH$ which is the pullback of the generator of $H^2(\mathrm{BSL}_2(\bR);\bR)$ via the map $\theta$. The class $e+v$ induces a map
\[
e+v: \BH\to K(\bR,2).
\]
Let $\widetilde{\BH}$ denote the homotopy fiber of the above map. Thus, there is a homotopy commutative diagram
\begin{equation}\label{eq0}
\begin{tikzpicture}[node distance=2.2cm, auto]
  \node (A) {$\widetilde{\BH}$};
  \node (C) [right of=A, node distance=1.5cm] {$\BH$};
  \node (B) [right of=C] {$K(\bR,2)$};
  \node (D) [below of=C, node distance=1.5cm] {$\mathrm{BSL}_2(\bR),$};
  \draw [->] (A) to node {$$}(C);
    \draw [->] (C) to node {$e+v$}(B);
  \draw [->] (C) to node {$\theta$}(D);
  \draw [<-] (D) to node {$\beta$}(A);
\end{tikzpicture}
\end{equation}
where $\beta$ is the composition of the inclusion of the homotopy fiber and the map $\theta$.  We denote the homotopy fiber of $\beta$ by $\overline{\overline{\BH}}$. Let $\gamma$ be the tautological $2$-plane bundle over $\mathrm{BSL}_2(\bR)$. Let $\mathrm{MT}\theta$ and $\mathrm{MT}\beta$ denote the Thom spectrum of the virtual bundles $\theta^*(-\gamma)$ and $\beta^*(-\gamma)$ respectively.  Let $\Omega_{\bullet}^{\infty}\mathrm{MT}\theta$ and $\Omega_{\bullet}^{\infty}\mathrm{MT}\beta$ denote the base component of the infinite loop spaces associated to the Thom spectrum $\mathrm{MT}\theta$ and $\mathrm{MT}\beta$ respectively.
\begin{thm}\label{thm3}
There is a homotopy commutative diagram
\[
\begin{tikzpicture}[node distance=3.2cm, auto]
  \node (A) {$\mathrm{B}\deHam(\Sigma,\partial)$};
  \node (C) [below of=A, node distance=1.5cm] {$ \mathrm{BSymp}^{\delta}(\Sigma,\partial)$};
  \node (B) [right of=A] {$\Omega_{\bullet}^{\infty}\mathrm{MT}\beta$};
  \node (D) [below of=B, node distance=1.5cm] {$\Omega_{\bullet}^{\infty}\mathrm{MT}\theta,$};
  \draw [->] (C) to node {$$}(D);
   \draw [<-] (C) to node {$$}(A);
  \draw [<-] (D) to node {$$}(B);
  \draw [->] (A) to node {$$}(B);
\end{tikzpicture}
\]
where the horizontal maps are homology isomorphisms in the stable range as \Cref{thm1}.
\end{thm}

\subsection{Capping off the last boundary component}As we shall see in \Cref{closedsurface}, for symplectomorphisms capping off the last boundary component of a surface also induces homology isomorphisms up to the same range as \Cref{thm1}. In other words for an embedding of a closed $2$-disk $D^2$ into a closed surface $\Sigma$, let $\dSymp(\Sigma, \text{rel } D^2)$ denote those symplectomorphisms whose supports are away from the embedded disk. 
\begin{thm}\label{cappingoffsymp}
The inclusion $\dSymp(\Sigma, \text{rel } D^2)\to \dSymp(\Sigma)$ induces a the map
\[
H_*(\BdSymp(\Sigma, \text{rel } D^2);\bZ)\to H_*(\BdSymp(\Sigma);\bZ),
\]
which is an isomorphism in the same range as \Cref{thm1}.
\end{thm}
 However, for extended Hamiltonians, we show that 
\[
H_*(\mathrm{B}\widetilde{\mathrm{Ham}}^{\delta}(\Sigma, \text{rel }D^2);\bZ)\to H_*(\mathrm{B}\widetilde{\mathrm{Ham}}^{\delta}(\Sigma);\bZ),
\]
cannot be an isomorphism in any range, in fact we show that $\widetilde{\mathrm{Ham}}^{\delta}(\Sigma, \text{rel }D^2)$ and $\widetilde{\mathrm{Ham}}^{\delta}(\Sigma)$ have different $H_1$ and $H_2$.

Nonetheless, for a closed surface $\Sigma$, we shall  describe below the difference between the homology of $\mathrm{B}\deHam(\Sigma)$ and the homology of $\mathrm{B}\deHam(\Sigma,\text{rel }D^2)$ in the same range as \Cref{thm1}. It is well known that the classifying space of an abelian group inherits the structure of a topological abelian group.  In particular  $\mathrm{B}\bR^{\delta}$ is a topological group and we shall show that it acts on $\mathrm{MT}\beta$ and the homotopy quotient of this action $\mathrm{B}\bR^{\delta}\hker\Omega^{\infty}\mathrm{MT}\beta$ describes the homology of $\mathrm{B}\deHam(\Sigma)$ in a range.
\begin{thm}\label{thm3'}
For a closed surface $\Sigma$, there is a homotopy commutative diagram
\[
\begin{tikzpicture}[node distance=3.2cm, auto]
  \node (A) {$\mathrm{B}\deHam(\Sigma,\text{\textnormal{rel} }D^2))$};
  \node (C) [below of=A, node distance=1.5cm] {$ \mathrm{B}\deHam(\Sigma)$};
  \node (B) [right of=A] {$\Omega^{\infty}\mathrm{MT}\beta$};
  \node (D) [below of=B, node distance=1.5cm] {$ \mathrm{B}\bR^{\delta}\hker\Omega^{\infty}\mathrm{MT}\beta,$};
  \draw [->] (C) to node {$$}(D);
   \draw [<-] (C) to node {$$}(A);
  \draw [<-] (D) to node {$$}(B);
  \draw [->] (A) to node {$$}(B);
\end{tikzpicture}
\]
where the horizontal maps in the same range as \Cref{thm1} induce  homology isomorphisms onto the connected components that they hit (see \Cref{thm10} for a geometric meaning of this theorem).
\end{thm}
\begin{cor}\label{cor1}
The map induced by capping off the last boundary component
\[
H_*(\mathrm{B}\deHam(\Sigma,\text{\textnormal{rel} }D^2);\bF_p)\to H_*(\mathrm{B}\deHam(\Sigma);\bF_p),
\]
is an isomorphism on homology with finite coefficients in the stable range.
\end{cor}
Note that for closed surfaces $\Sigma$ and $\Sigma'$, there is no comparison map from $\mathrm{B}\deHam(\Sigma)$ to $\mathrm{B}\deHam(\Sigma')$, but using \Cref{cor1}, one can find a zig-zag of isomorphisms between $\bF_p$-homology of $\mathrm{B}\deHam(\Sigma)$ and $\mathrm{B}\deHam(\Sigma')$ in the stable range of the surface with lower genus. 

For homology with $\bQ$-coefficients, however, we use a different zig-zag of isomorphisms to show that  
\begin{thm}\label{cor2'}
 The groups $H_*(\mathrm{B}\deHam(\Sigma);\bQ)$ and $H_*(\mathrm{B}\deHam(\Sigma');\bQ)$ are isomorphic for $*\leq \text{min}((2g(\Sigma)-2)/3,(2g(\Sigma')-2)/3)$.
\end{thm}
\begin{rem}
The isomorphism is given by a zig-zag of maps and in fact it induces an isomorphism in the same range for any coefficient subring  of $\bQ$ in which the Euler numbers $\chi(\Sigma)$ and $\chi(\Sigma')$ are invertible.
\end{rem}

\subsection{Characteristic classes of flat symplectic bundles} Recall $\mathrm{MCG}(\Sigma,\partial)$ denote the mapping class group of the surface $\Sigma$ fixing the boundary pointwise. As a result of Moser's theorem (\cite{MR0182927}), the topological groups $\Diff(\Sigma,\partial)$ and $\Symp(\Sigma,\partial)$ have the same group of connected components, thus we have the following short exact sequences of groups
\[
1\to\dSymp_0(\Sigma,\partial)\to\dSymp(\Sigma,\partial)\to \mathrm{MCG}(\Sigma,\partial)\to 1,
\]
\[
1\to \dHam(\Sigma,\partial)\to \deHam(\Sigma,\partial)\to \mathrm{MCG}(\Sigma,\partial)\to 1.
\]
\begin{rem}
In fact, there are uncountably different ways to extend the mapping class group by the Hamiltonian group for a surface with boundary (see \cite{MR2826937}[Theorem 7.2]). But for a closed surface $\Sigma$ the extension $\deHam(\Sigma)$ is unique. We consider the restriction of this unique extension to obtain the extended Hamiltonian group for surfaces with boundary. 
\end{rem}
Morita showed in \cite{morita1987characteristic} that the MMM-classes $\kappa_i\in H^{2i}(\mathrm{MCG}(\Sigma,\partial);\bZ)$ which are characteristic classes for surface bundles (see \Cref{sec3} for a definition of these classes) are nonzero in the stable range and even more the natural map
\[
\bZ[\kappa_1,\kappa_2,\cdots]\to H^*(\mathrm{MCG}(\Sigma,\partial);\bZ),
\]
is injective in the same stable range as \Cref{thm1}. We prove the same holds for flat symplectic surface bundles.
\begin{thm}\label{thm5}
The natural map induced by pulling back the MMM-classes to $\BdSymp(\Sigma,\partial)$ 
\[
\bZ[\kappa_1,\kappa_2,\cdots]\to H^*(\BdSymp(\Sigma,\partial);\bZ),
\]
is injective in the stable range.
\end{thm}
The situation is quite different with rational coefficients. The Bott vanishing theorem implies that $\kappa_i$ for $i>2$ vanishes in $H^*(\dSymp(\Sigma,\partial);\bQ)$. On the other hand,  Kotschick and Morita in \cite{kotschick2005signatures} proved that powers of $\kappa_1$ are nonzero in $H^*(\dSymp(\Sigma,\partial);\bQ)$. The (non)vanishing of $\kappa_2$ in the rational cohomology of $\dSymp(\Sigma,\partial)$ is not yet known. However, we prove all MMM-classes vanish in the cohomology of $\deHam(\Sigma,\partial)$ with real coefficients.
\begin{thm}\label{thm4}
The natural map 
\[
\bR[\kappa_1,\kappa_2,\cdots]\to H^*(\BdeHam(\Sigma,\partial);\bR),
\]
is a zero map.
\end{thm}
To give \Cref{thm3'} a geometric meaning, for a closed surface $\Sigma$, let 
\[
\begin{tikzpicture}[node distance=2cm, auto]
  \node (A) {$ \Sigma$};
  \node (B) [ right of=A] {$\Sigma\hcoker \deHam(\Sigma)$};
  \node (D) [below of=B, node distance=1.3cm] {$\BdeHam(\Sigma),$};
  \draw [->] (A) to node {$$} (B);
  \draw [->] (B) to node {$\pi$} (D);
\end{tikzpicture}
\]
denote the universal $\Sigma$-bundle whose holonomy lies in $\deHam(\Sigma)$. It is not hard to use the perfectness of $\dHam(\Sigma)$ (see \cite{MR1445290}) to show that the first MMM-class $\kappa_1$ is nonzero in $H^2(\BdeHam(\Sigma);\bZ)$. Consider the following map induced by the first MMM-class
\begin{equation}\label{map}
\frac{\kappa_1}{4-4g(\Sigma)}: \BdeHam(\Sigma)\to K(\bR,2).
\end{equation}
\begin{thm}\label{thm10}
There is a map
\[
\mathrm{B}\deHam(\Sigma,\text{\textnormal{rel} }D^2)\to \text{\textnormal{hofib}}(\frac{\kappa_1}{4-4g(\Sigma)}),
\]
that induces a homology isomorphism in the stable range.
\end{thm}

In order to find new invariants of flat symplectic surface bundles, we use \Cref{thm3} and existence of nontrivial cohomology classes in $H^*(\BH;\bZ)$ to prove that $H_*(\dSymp(\Sigma,\partial);\bZ)$ is highly nontrivial.  Note that any class in $H_2(\dSymp(\Sigma,\partial);\bZ)$  can be represented by a map
\[
f: \Sigma'\to \BdSymp(\Sigma,\partial),
\]
where $\Sigma'$ is a surface. The map $f$ induces a flat symplectic bundle  $E\to\Sigma'$ whose fibers are diffeomorphic to $\Sigma$. Since $E$ admits a codimension $2$ foliation with a transverse volume form, invariant under the holonomy, this foliation induces a well-defined map up to homotopy
\[
g:E\to \BH.
\]
Hence, one can easily see that this assignment defines a well-defined map from $H_2(\BdSymp(\Sigma,\partial);\bZ)$ to $H_4(\BH;\bZ)$.
\begin{thm}\label{thm2'}
The natural map
\[
H_2(\BdSymp(\Sigma,\partial);\bZ[\frac16])\to H_4(\BH;\bZ[\frac16]),
\]
is an isomorphism for $g(\Sigma)\geq 4$ and epimorphism for $g(\Sigma)\geq 3$.
\end{thm}
Kotschick and Morita used the extended flux homomorphism to construct a surjection map

\[
\begin{tikzcd}
H_2(\BdSymp(\Sigma,\partial);\bZ) \arrow[two heads, ""]{r} & \bZ\oplus S^2_{\bQ}\bR.
\end{tikzcd}
\]
where $S^2_{\bQ}\bR$ is the second symmetric power of $\bR$ as a $\bQ$-vector space. They asked in \cite[Problem 23]{kotschick2004characteristic} if this map is an isomorphism.  One can use \Cref{thm2'} to partially answer their problem, as we shall briefly explain here  (see \cref{cor2} for more precise statement). 
\begin{thm}\label{ttt}
There exists a certain homomorphism
\[
d: \bR\oplus (\bR\otimes\bR)\to H_4(\overline{\overline{\BH}};\bQ),
\]
so that for a surface $\Sigma$ with $g(\Sigma)\geq 4$, we have a short exact sequence
\[
0\to \Coker(d)\to H_2(\BdSymp(\Sigma,\partial);\bQ)\to \bQ\oplus S^2_{\bQ}\bR\to 0. 
\]
\end{thm} 
\begin{rem}
Hence, Kotschick-Morita's problem for a surface of genus larger than $4$ is equivalent to showing $\Coker(d)=0$. Given our state of knowledge about foliations with transverse volume form, proving that $d$ is surjective seems to be very hard!
\end{rem}
Given \Cref{ttt}, we obtain that there is a surjective map
\[
\begin{tikzcd}
H_2(\Omega_{\bullet}^{\infty}\mathrm{MT}\theta;\bQ) \arrow[two heads, ""]{r} & \bQ\oplus S^2_{\bQ}\bR.
\end{tikzcd}
\]
Since $H_*(\Omega_{\bullet}^{\infty}\mathrm{MT}\theta;\bQ)$ is a simply connected Hopf algebra over rationals, we deduce that 
\[
\begin{tikzcd}
H_{2k}(\Omega_{\bullet}^{\infty}\mathrm{MT}\theta;\bQ) \arrow[two heads, ""]{r} & S^k(\bQ\oplus S^2_{\bQ}\bR).
\end{tikzcd}
\]
Hence, we obtain a different proof of the main theorem of Kotschick and Morita in \cite{kotschick2004characteristic}:
\begin{cor}
There is a surjective map
\[
\begin{tikzcd}
H_{2k}(\BdSymp(\Sigma,\partial);\bQ) \arrow[two heads, ""]{r} & S^k(\bQ\oplus S^2_{\bQ}\bR).
\end{tikzcd}
\]
for $g(\Sigma)\geq 3k$.
\end{cor}
\subsection*{Outline} The paper is organized as follows: In \Cref{sec2}, we use McDuff's work on the volume preserving diffeomorphisms and Randal-Williams' work on homological stability for tangential structures to describe the group homology of $\dSymp(\Sigma,\partial)$ and $\deHam(\Sigma,\partial)$ in a range depending on the genus. In \Cref{sec3}, we study characteristic classes of surface bundles whose holonomy groups are area preserving which in particular leads us to give a homotopy theoretic proof of Kotschick-Morita's theorem \cite[Theorem 4]{kotschick2004characteristic} and partially answers their problem in \cite[Problem 23]{kotschick2004characteristic}.
\subsection*{Acknowledgment} I would like to thank S\o ren Galatius and Oscar Randal-Williams for many helpful discussions regarding this project. In the first formulation of \Cref{thm10}, there was a different description of the map \ref{map}. I am grateful to  Shigeyuki Morita who simplified my description as a multiple of the first MMM-class.  I am also thankful for the hospitality of the topology group in the Mathematical Institute of Universit{\"a}t M{\"u}nster during the period that this project was done. I would like to thank the referees for their comments that helped me to improve the presentation of my arguments.
\section{Homological stability}\label{sec2} In this section, we use the work of McDuff on volume preserving diffeomorphisms (\cite{mcduff1983local,MR707329}) and Randal-Williams' work  on homological stability of moduli spaces (\cite{randal2009resolutions})  to prove \Cref{thm1} and \Cref{thm2}.

 Let $(\Sigma,\omega)$ be a pair consisting of a surface $\Sigma$ with a nowhere zero $2$-form $\omega$ on $\Sigma$. Let $\Symp_{\omega}(\Sigma,\partial)$ denote the group of compactly supported $\omega$-preserving diffeomorphisms of the interior of $\Sigma$. Let $(\Sigma',\omega')$ be a pair where $\Sigma\subset \Sigma'$ is a subsurface and $\omega=\omega'|_{\Sigma}$ is the restriction of the volume form to $\Sigma$. There is a natural group homomorphism
\[
s(\Sigma,\Sigma'): \dSymp_{\omega}(\Sigma,\partial)\to \dSymp_{\omega'}(\Sigma',\partial)
\]
which is given by extending $\omega$-preserving diffeomorphisms of $\Sigma$ by the identity over $\Sigma'\backslash \Sigma$. \Cref{thm1} can be reformulated as follows
\begin{thm*}
The map
\[
H_*(\BdSymp_{\omega}(\Sigma,\partial);\bZ) \to H_*(\BdSymp_{\omega'}(\Sigma',\partial);\bZ)
\]
induced by $s(\Sigma,\Sigma')$ is an isomorphism for $*\leq (2g(\Sigma)-2)/3$ and epimorphism for $*\leq 2g(\Sigma)/3$.
\end{thm*}
Let $\Sigma$ be a surface with boundary. We treat the case where $\Sigma$ is a closed surface separately in \Cref{{closedsurface}}.  Given the observation of Kotschick and Morita  in \cite [Section 2.1]{kotschick2005signatures}, that the group $\dSymp_{\omega}(\Sigma,\partial)$ is perfect for $g(\Sigma)\geq 3$, we can consider the Quillen plus construction of $\BdSymp_{\omega}(\Sigma,\partial)$ for $g(\Sigma)\geq 3$. As we shall see there exists a model for the plus construction of $\BdSymp_{\omega}(\Sigma,\partial)$ to which the general homological stability theorem in \cite[Theorem 7.1]{randal2009resolutions} can be applied.  To describe this model, we first need to recall a theorem due to McDuff.

\subsection{Recollection from McDuff's work on  volume preserving diffeomorphisms} Let $\overline{\BSymp_{\omega}(\Sigma,\partial)}$ denote the homotopy fiber of the map
\[
\BdSymp_{\omega}(\Sigma,\partial)\to \BSymp_{\omega}(\Sigma,\partial)
\]
induced by the identity homomorphism. 
\begin{rem} In fact, in this case we can describe the homotopy fiber more concretely. Recall from the introduction that $\Symp_{\omega}(\Sigma,\partial)\simeq \text{MCG}(\Sigma,\partial)$. Hence, we have the following fiber sequence
\[
\BdSymp_{0,\omega}(\Sigma,\partial)\to \BdSymp_{\omega}(\Sigma,\partial)\to \mathrm{B}\text{MCG}(\Sigma,\partial),
\]
where $\Symp_{0,\omega}(\Sigma,\partial)$ is the identity component of the topological group $\Symp_{\omega}(\Sigma,\partial)$. We obtain a map
\[
\overline{\BSymp_{\omega}(\Sigma,\partial)}\to \BdSymp_{0,\omega}(\Sigma,\partial)
\]
which is a homotopy equivalence. 
\end{rem}
The action of $\dSymp_{\omega}(\Sigma,\partial)$ on $\Sigma$ gives the following surface bundle
\[
 \begin{tikzpicture}[node distance=2cm, auto]
  \node (A) {$ \Sigma$};
  \node (B) [ right of=A] {$\Sigma\hcoker \dSymp_{\omega}(\Sigma,\partial)$};
  \node (D) [below of=B, node distance=1.3cm] {$\BdSymp_{\omega}(\Sigma,\partial),$};
  \draw [->] (A) to node {$$} (B);
  \draw [->] (B) to node {$\pi$} (D);
\end{tikzpicture}
\]
whose holonomy group is $\dSymp_{\omega}(\Sigma,\partial)$. Therefore it is a foliated (flat) $\Sigma$-bundle whose holonomy preserves the volume form of the fibers. The normal bundle to the foliation is the vertical tangent bundle of $\pi$. By the general theory of Haefliger (\cite{haefliger1971homotopy}), the foliation on the total space induces a map well defined up to homotopy
\[
\Sigma\hcoker \dSymp_{\omega}(\Sigma,\partial)\to \BH.
\]
 If we pull back this foliated bundle to $\overline{\BSymp_{\omega}(\Sigma,\partial)}$, we obtain the product bundle
 \[
 \begin{tikzpicture}[node distance=2cm, auto]
  \node (A) {$ \overline{\BSymp_{\omega}(\Sigma,\partial)}\times \Sigma$};
  \node (B) [node distance=3.2cm, right of=A] {$\Sigma\hcoker \dSymp_{\omega}(\Sigma,\partial)$};
  \node (C)[below of=A, node distance=1.3cm] {$ \overline{\BSymp_{\omega}(\Sigma,\partial)}$};
  \node (D) [below of=B, node distance=1.3cm] {$\BdSymp_{\omega}(\Sigma,\partial),$};
  \draw [->] (A) to node {$$} (B);
    \draw [->] (A) to node {$$} (C);
      \draw [->] (C) to node {$$} (D);
  \draw [->] (B) to node {$$} (D);
\end{tikzpicture}
\]
 with a foliation $\mathcal{F}$  transverse to the fibers $\{ x\}\times \Sigma$ (see \cite{mcduff1983local,MR707329} for more details). Since this bundle is trivial, the normal bundle to the foliation $\mathcal{F}$ is induced by the pull back of the tangent bundle $T\Sigma$ via the projection  $\overline{\BSymp_{\omega}(\Sigma,\partial)}\times \Sigma\to \Sigma$. Hence, we have the following homotopy commutative diagram
\begin{equation}\label{eq1}
\begin{gathered}
 \begin{tikzpicture}[node distance=2cm, auto]
  \node (A) {$ \overline{\BSymp_{\omega}(\Sigma,\partial)}\times \Sigma$};
  \node (B) [node distance=3.4cm, right of=A] {$\mathrm{B}\Gamma_{2}^{\text{vol}}$};
  \node (C) [below of=A, node distance=1.3cm] {$\Sigma$};  
  \node (D) [below of=B, node distance=1.3cm] {$\mathrm{BSL}_2(\bR).$};
  \draw[->] (C) to node {$\tau$} (D);
  \draw [->] (A) to node {$F$} (B);
    \draw [->] (A) to node {$$} (C);
  \draw [->] (B) to node {$\theta$} (D);
\end{tikzpicture}
\end{gathered}
\end{equation}
For the point-set model of the diagram \ref{eq1} see \cite{mcduff1979foliations} and \cite[Section 5.1]{nariman2014homologicalstability}. Let $\text{Sect}(\Sigma)$ be the space of sections of $\tau^*(\theta)$, the pullback of  $\mathrm{B}\Gamma_{2}^{\text{vol}}$ over $\Sigma$. After 
choosing a base section $s_0$ for $\text{Sect}(\Sigma)$, one can define $\text{Sect}(\Sigma,\partial)$ to be those sections that are equal to $s_0$ in the germ of the boundary (In fact in the point-set model, there is a canonical base section $s_0$ defined by the foliation by points on the surface). For any $x\in  \overline{\BSymp_{\omega}(\Sigma,\partial)}$, the map $F|_{x\times \Sigma}$ is a lifting of the map $\tau$ to $\mathrm{B}\Gamma_{2}^{\text{vol}}$, hence we obtain a map
\[
f_{\Sigma}: \overline{\BSymp_{\omega}(\Sigma,\partial)}\to \text{Sect}(\Sigma,\partial).
\]
The section space $\text{Sect}(\Sigma,\partial)$ is not connected and in fact $\pi_0(\text{Sect}(\Sigma,\partial))=\bR$ which is given by the integration of $\omega$ over the surface. Let $\text{Sect}_0(\Sigma,\partial)$ denote the base component.
\begin{thm}[McDuff \cite{MR707329}]\label{mcduff}
The map
\[
f_{\Sigma}: \overline{\BSymp_{\omega}(\Sigma,\partial)}\to \text{\textnormal{Sect}}_0(\Sigma,\partial).
\]
induces a homology isomorphism. 
\end{thm}
Using obstruction theory, one can show that $\pi_1(\text{Sect}_0(\Sigma,\partial))$ is a nilpotent group and sits in a short exact sequence
\[
0\to \bR\to \pi_1(\text{Sect}_0(\Sigma,\partial))\to H^1(\Sigma,\partial;\bR)\to 0.
\]
Hence we have a map
\[
h: \text{Sect}_0(\Sigma,\partial)\to \mathrm{B}H^1(\Sigma,\partial;\bR).
\]
Let $\widetilde{\text{Sect}_0(\Sigma,\partial)}$ be the homotopy fiber of $h$. In fact, McDuff obtained \Cref{mcduff} as a corollary of the following 
\begin{thm}\label{mcduff1}
There is a map
\[
\tilde{f_{\Sigma}}: \BdHam_{\omega}(\Sigma,\partial)\to \widetilde{\text{\textnormal{Sect}}_0(\Sigma,\partial)}
\]
that induces a homology isomorphism. 
\end{thm}
\subsection{The tangential $\theta$-structures}\label{tang} To describe a point-set model for the section space on which $\Symp_{\omega}(\Sigma,\partial)$ acts, we shall recall the space of tangential structures.  Let $B$ be any topological space. For a map $\alpha: B\to \mathrm{BSL}_2(\bR)$ that is a fibration, let $\text{Bun}_{\partial}(T\Sigma,{\alpha}^*\gamma)$ denote the space of all bundle maps $T\Sigma\rightarrow {\alpha}^*\gamma$ from the tangent bundle of $\Sigma$ to ${\alpha}^*(\gamma)$ that are standard on a germ of the boundary and equipped with the compact-open topology (See  \cite[Section 1.1]{randal2009resolutions} for what it means to be standard near the boundary). The whole diffeomorphism group  $\Diff(\Sigma,\partial)$ acts on bundle maps by precomposing a bundle map with the differential of a diffeomorphism and we shall restrict this action to the volume preserving diffeomorphisms. 
\begin{defn}\label{moduli}
The moduli space of $\alpha$-tangential structure $\mathcal{M}^{\alpha}(\Sigma,\partial)$ is defined to be
\[
\text{Bun}_{\partial}(T\Sigma,{\alpha}^*\gamma)\hcoker \Symp_{\omega}(\Sigma,\partial).
\] 
\end{defn}

Now consider the tangential structure $\theta:\BH\to \mathrm{BSL}_2(\bR)$. Recall  $\gamma$ is the tautological $2$-plane bundle over $\mathrm{BSL}_2(\bR)$. One can define a map between $ \text{Sect}(\Sigma,\partial)$ and $\text{Bun}_{\partial}(T\Sigma,{\theta}^*\gamma)$ as follows. First fix an isomorphism between $T\Sigma$ and $\tau^*\gamma$. Every section $s\in   \text{Sect}(\Sigma,\partial)$ gives a map $s:\Sigma\to \BH$ such that ${\theta}\circ s=\tau$. Hence, we obtain a bundle map $s^*: T\Sigma\to {\theta}^*\gamma$. It is easy to prove that  the map that associates a bundle map to a section
 \[
 \text{Sect}(\Sigma,\partial)\to \text{Bun}_{\partial}(T\Sigma,{\theta}^*\gamma),
 \]
 is a weak homotopy equivalence (see \cite[Section 3.2]{MR3044208}). The advantage of the section space model is that it is easier to study its homotopy type in this model. On the other hand, a priori there is no natural action of $\Symp_{\omega}(\Sigma,\partial)$ on the space of sections.  But as we discussed above, there is a natural action of $\Symp_{\omega}(\Sigma,\partial)$ on the bundle maps. Because we want to keep track of actions, we use the bundle maps model and when we want to calculate its homotopy groups, we use the section space model.

For a volume preserving embedding $(\Sigma,\omega)\hookrightarrow (\Sigma',\omega')$, the canonical base section $s'_0\in \text{Sect}(\Sigma',\partial)$ restricts to the base section $s_0\in \text{Sect}(\Sigma,\partial)$, thus we have a map
\[
\text{Bun}_{\partial}(T\Sigma,{\theta}^*\gamma)\to \text{Bun}_{\partial}(T\Sigma',{\theta}^*\gamma),
\]
that is equivariant with respect to the inclusion $\Symp_{\omega}(\Sigma,\partial)\to \Symp_{\omega'}(\Sigma',\partial)$. Hence, we obtain a stabilization map
\[
\mathcal{M}^{\theta}(\Sigma,\partial)\to \mathcal{M}^{\theta}(\Sigma',\partial),
\]
by extending over $\Sigma'\backslash\Sigma$ via the base section. 

 To prove \Cref{thm1}, we  relate the plus construction of $\BdSymp_{\omega}(\Sigma,\partial)$ to $\mathcal{M}^{\theta}(\Sigma,\partial)$  and then use the main theorem in \cite[Theorem 7.1]{randal2009resolutions}.  
%

\begin{proof}[Proof of \Cref{thm1}] Let $\overline{\BSymp_{\omega}(\Sigma,\partial)}$ be the homotopy fiber of the map 
\[
\iota: \BdSymp_{\omega}(\Sigma,\partial)\to \BSymp_{\omega}(\Sigma,\partial).
\]
A model for this homotopy fiber on which the group $\Symp_{\omega}(\Sigma,\partial)$ acts, is the pullback of the universal $\Symp_{\omega}(\Sigma,\partial)$-bundle $$\mathrm{E}\Symp_{\omega}(\Sigma,\partial)\to \BSymp_{\omega}(\Sigma,\partial),$$ via the map $\iota$. This pullback is a principal $\Symp_{\omega}(\Sigma,\partial)$- bundle over the base space $\BdSymp_{\omega}(\Sigma,\partial)$, hence the group $\Symp_{\omega}(\Sigma,\partial)$ acts on it.  Note that given that this action is free, the homotopy quotient of this action 
\[
\overline{\BdSymp_{\omega}(\Sigma,\partial)}\hcoker \Symp_{\omega}(\Sigma,\partial),
\]
is homotopy equivalent to the quotient of the action which is homeomorphic to the base space where in this case is homotopy equivalent to $\BdSymp_{\omega}(\Sigma,\partial)$. 
As explained in \cite[Section 5.1]{nariman2014homologicalstability} or \cite[Section 1.2.2]{nariman2015stable}, there is a point-set model for the map in \Cref{mcduff}
\[
f_{\Sigma}: \overline{\BSymp_{\omega}(\Sigma,\partial)}\to \text{Bun}_{\partial}(T\Sigma,{\theta}^*\gamma),
\]
that is $\Symp_{\omega}(\Sigma,\partial)$-equivariant. Hence, we obtain a map
\[
\overline{\BSymp_{\omega}(\Sigma,\partial)}\hcoker \Symp_{\omega}(\Sigma,\partial)\to \mathcal{M}^{\theta}(\Sigma,\partial).
\]
Recall that $\text{Bun}_{\partial}(T\Sigma,{\theta}^*\gamma)$ is not connected but the action of $\Symp_{\omega}(\Sigma,\partial)$ preserves the connected components. Let $\mathcal{M}_0^{\theta}(\Sigma,\partial)$ be the base component. Thus, McDuff's theorem implies that there is a zig-zag of maps
\[
\BdSymp_{\omega}(\Sigma,\partial)\leftarrow \overline{\BdSymp_{\omega}(\Sigma,\partial)}\hcoker \Symp_{\omega}(\Sigma,\partial) \to \mathcal{M}_0^{\theta}(\Sigma,\partial),
\]
that are homology isomorphisms. By naturality of our constructions, it is easy to see that the above homology isomorphisms commute with the stabilization maps. 

To prove that $\mathcal{M}_0^{\theta}(\Sigma,\partial)$ exhibits homological stability, we show that $\pi_0(\mathcal{M}^{\theta}(\Sigma,\partial))$ is stable as the genus increases and $\mathcal{M}^{\theta}(\Sigma,\partial)$ is also homologically stable. Randal-Williams' theorem (\cite[Theorem 7.1]{randal2009resolutions}), however, implies that $\mathcal{M}^{\theta}(\Sigma,\partial)$ exhibits homological stability if the connected components $\pi_0(\mathcal{M}^{\theta}(\Sigma,\partial))$ stabilize with respect to the genus. Therefore, we only need to show that $\pi_0(\mathcal{M}^{\theta}(\Sigma,\partial))$ is stable. To do so, consider the fibration
\[
\text{Bun}_{\partial}(T\Sigma,{\theta}^*\gamma)\to \mathcal{M}^{\theta}(\Sigma,\partial)\to \BSymp_{\omega}(\Sigma,\partial).
\]
The relevant part of the long exact sequence of the homotopy groups is
   \[
  \pi_1(\BSymp_{\omega}(\Sigma,\partial))\to\pi_0(\text{Bun}_{\partial}(T\Sigma,{\theta}^*\gamma))\to\pi_0(\mathcal{M}^{\theta}(\Sigma,\partial))\to\pi_0(\BSymp_{\omega}(\Sigma,\partial)).
   \]
Since $\Sigma$ has a boundary, the tangent bundle can be trivialized. Therefore, in the section space model for the bundle maps $\text{Bun}_{\partial}(T\Sigma,{\theta}^*\gamma)$ the map $\tau$ is null-homotopic. Hence, $\mathrm{Sect}(\Sigma,\partial)$ is in fact homotopy equivalent to the mapping space $\mathrm{Map}_{\partial}(\Sigma,\overline{\mathrm{B}\Gamma_2^{\text{vol}}})$ where $\overline{\mathrm{B}\Gamma_2^{\text{vol}}}$ is the homotopy fiber of the map $\theta$. Given that $\Sigma$ is $2$ dimensional and $\overline{\mathrm{B}\Gamma_2^{\text{vol}}}$ is simply connected (\cite{mcduff1981groups}), we have $$\pi_0(\text{Bun}_{\partial}(T\Sigma,{\theta}^*\gamma))=\pi_0(\mathrm{Map}_{\partial}(\Sigma,\overline{\mathrm{B}\Gamma_2^{\text{vol}}}))=H_2(\overline{\mathrm{B}\Gamma_2^{\text{vol}}}; \bZ).$$ The volume form induces a map 
\[
\bar{v}: \overline{\mathrm{B}\Gamma_2^{\text{vol}}}\to K(\bR,2),
\]
which is known from \cite{mcduff1981groups} to be $3$-connected. Therefore $H_2(\overline{\mathrm{B}\Gamma_2^{\text{vol}}}; \bZ)\cong \bR$. More concretely,  let  $f_1$ and $f_2$ be bundle maps in  $\text{Bun}_{\partial}(T\Sigma,{\theta}^*\gamma)$. We consider them as lifts of the tangent bundle to $\mathrm{B}\Gamma_2^{\text{vol}}$. They are in the same component of $\text{Bun}_{\partial}(T\Sigma,{\theta}^*\gamma)$ if the volumes of surface $\Sigma$ given by the two forms $f_i^*(\bar{v})$ are equal. Given that the every of the mapping class group $\text{MCG}(\Sigma,\partial)$ can be realized as a volume preserving diffeomorphism, the action of the mapping class group $\text{MCG}(\Sigma,\partial)$ on the set of components is trivial. Therefore we have $$\pi_0(\mathcal{M}^{\theta}(\Sigma,\partial))=\pi_0(\text{Bun}_{\partial}(T\Sigma,{\theta}^*\gamma))=H_2(\overline{\mathrm{B}\Gamma_2^{\text{vol}}}; \bZ)=\bR.$$ Thus,  the connected components of $\mathcal{M}^{\theta}(\Sigma)$ stabilize and the stabilization map or gluing surfaces along the boundary components corresponds to the addition of classes in $H_2(\overline{\mathrm{B}\Gamma_2^{\text{vol}}})$. 

To find a stability range, Randal-Williams defined a notion of $k$-triviality \cite[Definition 6.2]{randal2009resolutions} and proved that if a $\theta$-structure stabilizes at genus $h$, then it would be $(2h+1)$-trivial. Since $\theta$-structure stabilizes at genus $0$, by \cite[Theorem 7.1]{randal2009resolutions} the stability range for stabilization maps  is the same as the stability range for the orientation structure $\mathrm{BSO}(2)\to \mathrm{BO}(2)$. Therefore, the groups $\dSymp_{\omega}(\Sigma,\partial)$ have the same stability range as the mapping class groups.
\end{proof}
\begin{proof}[Proof of \Cref{thm2}] 
To recall the setup, let $(\Sigma,\omega)\hookrightarrow (\Sigma', \omega')$ be a volume preserving embedding such that the volumes of $\Sigma$ and $\Sigma'$ are normalized to be the negative of the Euler numbers respectively. This volume preserving embedding induces a stabilization map
\[
H_*(\BdeHam_{\omega}(\Sigma,\partial);\bZ)\to H_*(\BdeHam_{\omega'}(\Sigma',\partial);\bZ),
\]
which we want to prove is an isomorphism for $*\leq (2g(\Sigma)-2)/3$ and epimorphism for $*\leq 2g(\Sigma)/3$. To do so, similar to the proof of \Cref{thm1}, we show that $\BdeHam_{\omega}(\Sigma,\partial)$ is homology equivalent to a moduli space of a tangential structure whose $\pi_0$ stabilizes with respect to the genus. Recall that the extended Hamiltonian group $\eHam_{\omega}(\Sigma,\partial)$ hits all the connected components of $\Symp_{\omega}(\Sigma,\partial)$ and similar to $\Symp_{\omega}(\Sigma,\partial)$,  has contractible connected components.  Note that the  group extension 
\[
1\to\dHam(\Sigma,\partial)\to \deHam(\Sigma,\partial)\to \text{MCG}(\Sigma,\partial)\to 1,
\]
induces the fibration sequence
\[
\BdHam(\Sigma,\partial)\to \BdeHam(\Sigma,\partial)\to \mathrm{B}\text{MCG}(\Sigma,\partial)\simeq \BSymp_{\omega}(\Sigma,\partial).
\]
Therefore, the space $\BdHam_{\omega}(\Sigma,\partial)$ is the homotopy fiber of the map
\[
\BdeHam_{\omega}(\Sigma,\partial)\to \BSymp_{\omega}(\Sigma,\partial),
\]
which is induced by the identity homomorphism $\deHam_{\omega}(\Sigma,\partial)\to \eHam(\Sigma,\partial)$ and then including into $\Symp_{\omega}(\Sigma,\partial)$. Hence, similar to the proof of \Cref{thm1}, there is a point-set model for $\BdHam(\Sigma,\partial)$ on which $\Symp_{\omega}(\Sigma,\partial)$ acts  and the induced map from the homotopy quotient of this action to the quotient space
\begin{equation}\label{11}
\BdHam_{\omega}(\Sigma,\partial)\hcoker \Symp_{\omega}(\Sigma,\partial)\xrightarrow{\simeq} \BdeHam_{\omega}(\Sigma,\partial),
\end{equation}
is a homotopy equivalence. 

 On the other hand by \Cref{mcduff1}, we have a homotopy commutative diagram
\begin{equation}\label{Y}
  \begin{gathered}
 \begin{tikzpicture}[node distance=2cm, auto]
  \node (A) {$\BdHam_{\omega}(\Sigma,\partial)$};
  \node (B) [below of=A]{$\widetilde{\text{\textnormal{Sect}}_0(\Sigma,\partial)}$};
  \node (C) [right of=A, node distance=3cm]{$\BdSymp_{0,\omega}(\Sigma,\partial)$};
  \node (D) [right of=B, node distance=3cm]{$ \text{\textnormal{Sect}}_0(\Sigma,\partial)$};
  \node (E) [right of=C, node distance=3.6cm]{$\mathrm{B}H^1(\Sigma,\partial;\bR)$};
  \node (F) [below of=E]{$\mathrm{B}H^1(\Sigma,\partial;\bR).$};
  \draw [->] (A) to node {$$} (C);
  \draw [->] (B) to node {$$} (D);
  \draw [->] (A) to node {$\tilde{f_{\Sigma}}$} (B);
  \draw [->] (C) to node {$f_{\Sigma}$} (D);
  \draw [->] (C) to node {BFlux} (E);
  \draw [->] (E) to node {$\cong$} (F);
  \draw [->] (D) to node {$h$} (F);
 \end{tikzpicture}
   \end{gathered}
 \end{equation}
The Flux map is $\Symp_{\omega}(\Sigma,\partial)-$equivariant (see \cite[Lemma 6]{kotschick2005signatures}). Given the appropriate point-set model for the section space $ \text{\textnormal{Sect}}_0(\Sigma,\partial)$ as the bundle maps, the maps $f_{\Sigma}$ and $h$ are also $\Symp_{\omega}(\Sigma,\partial)-$equivariant. In the claim below, we shall prove that there is a point-set model for  $\widetilde{\text{\textnormal{Sect}}_0(\Sigma,\partial)}$ given by certain bundle maps. Thus, using the same constructions as in \cite[Section 5.1]{nariman2014homologicalstability} or \cite[Section 1.2.2]{nariman2015stable}, one obtains a $\Symp_{\omega}(\Sigma,\partial)-$equivariant model for the map  $\tilde{f_{\Sigma}}$. Hence, we have 
\[
\BdeHam_{\omega}(\Sigma,\partial)\xleftarrow{\simeq} \BdHam_{\omega}(\Sigma,\partial)\hcoker \Symp_{\omega}(\Sigma,\partial)\to \widetilde{\text{\textnormal{Sect}}_0(\Sigma,\partial)}\hcoker \Symp_{\omega}(\Sigma,\partial),
\]
where the second arrow induces only a homology isomorphism. Hence to prove the theorem, the idea is to show that the space $\widetilde{\text{\textnormal{Sect}}_0(\Sigma,\partial)}$ is in fact the space of bundle maps of a certain tangential structure over the surface $\Sigma$.

Recall from the diagram \ref{eq0} in the introduction that $\widetilde{\BH}$ is the homotopy fiber of the map $e+v: \BH\to K(\bR,2)$ and there is a tangential structure $\beta: \widetilde{\BH}\to \mathrm{BSL}_2(\bR)$.

\noindent{\bf Claim:} There is a map 
\[
\widetilde{\text{\textnormal{Sect}}_0(\Sigma,\partial)}\to \text{Bun}_{\partial}(T\Sigma,{\beta}^*\gamma),
\]
which is a weak homotopy equivalence.

\noindent{\it Proof of the claim:} Let $ \text{Bun}_{\partial,0}(T\Sigma,{\theta}^*\gamma)$ denote the base point component of $ \text{Bun}_{\partial}(T\Sigma,{\theta}^*\gamma)$. We write $\text{Map}_{\partial}(\Sigma,K(\bR,2))$ to denote the continuous mappings that send the germs of the boundary to the base point of $K(\bR,2)$ and the base point of the space of maps is the constant map whose value is the base point of $K(\bR,2)$.  Let $\text{Map}_{\partial,0}(\Sigma,K(\bR,2))$ denote its base point component. For brevity, we denote $H^1(\Sigma,\partial;\bR)$ by $H^1_{\bR}$ which is also the fundamental group of the mapping space $\text{Map}_{\partial,0}(\Sigma,K(\bR,2))$. 

Recall that Thom's theorem (\cite{MR0089408})  says the space of maps from a topological space $X$ to the Eilenberg MacLane space $K(G,m)$ is homotopically equivalent to the products of Eilenberg-MacLane spaces $\prod_{i=0}^m K(H^{m-i}(X;G),i)$. Now for the mapping space $\text{Map}_{\partial,0}(\Sigma,K(\bR,2))$, since we are considering the base point component, we are omitting the factor  $K(H^2(\Sigma,\partial;\bR),0)$  in the Thom theorem. And since we are considering the subspace of maps that send the boundary to the base point of $K(\bR,2)$, we are omitting the factor $K(H^0(\Sigma,\partial;\bR),2)$ in the splitting in Thom's theorem. Therefore, the natural map from the mapping space to the classifying space of its fundamental group
\[
\text{Map}_{\partial,0}(\Sigma,K(\bR,2))\xrightarrow{\simeq} \mathrm{B}{H^1_{\bR}}^{\delta},
\]
 is a homotopy equivalence. Note that the fibration sequence
\[
\widetilde{\BH}\to \BH\xrightarrow{e+v} K(\bR,2),
\]
induces the following fibration
\begin{equation}\label{fib1}
\text{Bun}_{\partial}(T\Sigma,{\beta}^*\gamma)\to \text{Bun}_{\partial,0}(T\Sigma,{\theta}^*\gamma)\to \text{Map}_{\partial,0}(\Sigma,K(\bR,2)).
\end{equation}
For a surface with boundary, the above sequence is a fibration of mappings spaces, but we write bundle maps to keep track of the action of the group $\Symp_{\omega}(\Sigma,\partial)$ on spaces in the above fibration. Recall from \Cref{tang} that the map $ \text{\textnormal{Sect}}_0(\Sigma,\partial)\to \text{Bun}_{\partial,0}(T\Sigma,{\theta}^*\gamma)$ is a weak equivalence and that was how in the first place we defined the action of $\Symp_{\omega}(\Sigma,\partial)$ on the section spaces. Hence, to prove the claim, it is enough to show the following diagram is homotopy commutative
\begin{equation}\label{X}
  \begin{gathered}
 \begin{tikzpicture}[node distance=2cm, auto]
  \node (C) {$\BdSymp_{0,\omega}(\Sigma,\partial)$};
  \node (D) [below of=C]{$ \text{\textnormal{Sect}}_0(\Sigma,\partial)$};
  \node (E) [right of=C, node distance=4.2cm]{$\mathrm{B}{H^1_{\bR}}^{\delta}$};
  \node (F) [below of=E]{$\text{Map}_{\partial,0}(\Sigma,K(\bR,2)).$};
  \draw [->] (C) to node {$f_{\Sigma}$} (D);
  \draw [->] (C) to node {BFlux} (E);
  \draw [->] (F) to node {$\simeq$} (E);
  \draw [->] (D) to node {$-\circ (e+v)$} (F);
 \end{tikzpicture}
   \end{gathered}
 \end{equation}
To do so, let us recall how the map $f_{\Sigma}$ is defined. Consider the Borel construction $\Sigma\hcoker \dSymp_{0,\omega}(\Sigma,\partial)$ as a foliated surface bundle with a transverse volume form over $\BdSymp_0(\Sigma,\partial)$. Note that topologically this surface bundle is the trivial bundle $\BdSymp_{0,\omega}(\Sigma,\partial)\times \Sigma$, because topologically it is classified by a map to $\BSymp_{0,\omega}(\Sigma,\partial)$ which is contractible. Hence, by the general theory of foliations, we have a homotopy commutative diagram
\begin{equation}\label{eq3}
\begin{gathered}
 \begin{tikzpicture}[node distance=2cm, auto]
  \node (A) {$\BdSymp_{0,\omega}(\Sigma,\partial)\times \Sigma$};
  \node (B) [node distance=3.4cm, right of=A] {$\mathrm{B}\Gamma_{2}^{\text{vol}}$};
  \node (C) [below of=A, node distance=1.5cm] {$\Sigma$};  
  \node (D) [below of=B, node distance=1.5cm] {$\mathrm{BSL}_2(\bR).$};
  \node (E) [right of=B, node distance=3.cm] {$K(\bR,2)$};
  \draw[->] (C) to node {$\tau$} (D);
  \draw [->] (A) to node {$F$} (B);
    \draw [->] (A) to node {$\text{proj}$} (C);
  \draw [->] (B) to node {$\theta$} (D);
  \draw [->] (B) to node {$-\circ (e+v)$} (E);
\end{tikzpicture}
\end{gathered}
\end{equation}
Since the space $\text{Map}_{\partial,0}(\Sigma,K(\bR,2))\simeq \mathrm{B}{H^1_{\bR}}^{\delta}$ is an 
Eilenberg-MacLane space, to prove that the diagram \ref{X} is homotopy commutative, we need to show that the two maps $\text{BFlux}$ and $f_{\Sigma}\circ(-\circ (e+v))$ represent the same 
cohomology class in $H^1(\BdSymp_{0,\omega}(\Sigma,\partial); H^1_{\bR})$. 

Using the Kunneth theorem, $H^2(\BdSymp_{0,\omega}(\Sigma,\partial)\times \Sigma;\bR)$ is isomorphic to 
\begin{equation}\label{Kunneth}
H^2(\BdSymp_{0,\omega}(\Sigma,\partial);\bR)\oplus H^1_{\bR}\otimes H^1(\BdSymp_{0,\omega}(\Sigma,\partial);\bR)\oplus H^2(\Sigma,\partial; \bR).
\end{equation}
Note that the class represented by $f_{\Sigma}\circ(-\circ (e+v))$ is the same class obtained by projecting the class $$F^*(e+v)\in H^2(\BdSymp_{0,\omega}(\Sigma,\partial)\times \Sigma;\bR),$$ to the summand $H^1_{\bR}\otimes H^1(\BdSymp_{0,\omega}(\Sigma,\partial);\bR)$ in the decomposition (\ref{Kunneth}). Since the volume form is normalized, the restriction of $F^*(e+v)$ to each fiber is zero, therefore the projection of $F^*(e+v)$ to  $H^2(\Sigma,\partial;\bR)$ is zero. Given that the foliation on $\BdSymp_{0,\omega}(\Sigma,\partial)\times \Sigma$ is trivial near the boundary of the fibers, the map $F$ is constant near the boundary of the fibers, hence the projection of $F^*(e+v)$ to $H^2(\BdSymp_{0,\omega}(\Sigma,\partial);\bR)$ is also zero (this fact can be observed geometrically as the combination of \cite[Proposition 8]{kotschick2004characteristic} and \cite[Corollary 15]{kotschick2004characteristic}). Finally, by the calculation in \cite[Proposition 8]{kotschick2004characteristic} and \cite[Lemma 8]{kotschick2005signatures}, the projection of $F^*(e+v)$ to $H^1_{\bR}\otimes H^1(\BdSymp_{0,\omega}(\Sigma,\partial);\bR)$ is indeed the Flux. This finishes the proof of the claim. $ \blacksquare$

Now we can use the homological stability theorem for moduli space of tangential structures (\cite[Theorem 7.1]{randal2009resolutions}) to finish the proof of the theorem.  The moduli space of $\beta-$structures $\mathcal{M}^{\beta}(\Sigma,\partial)$ is defined to be $\text{Bun}_{\partial}(T\Sigma,{\beta}^*\gamma)\hcoker \Symp_{\omega}(\Sigma,\partial)$. Given that $\BdeHam_{\omega}(\Sigma,\partial)$ is homology equivalent to  $\mathcal{M}^{\beta}(\Sigma,\partial)$, the moduli space $\mathcal{M}^{\beta}(\Sigma,\partial)$ is connected. Therefore we have stability on $\pi_0(\mathcal{M}^{\beta}(\Sigma,\partial))$, hence similar to the argument in the proof of \Cref{thm1}, Randal-Williams' theorem applies and we deduce that the groups $\deHam_{\omega}(\Sigma,\partial)$ have the same stability range as the mapping class group. \end{proof}

\subsection{Closing the last boundary component}\label{closedsurface} Let $\iota: (D^2,\omega|_{D^2})\hookrightarrow (\Sigma,\omega)$ be a volume preserving embedding of a disk into a closed surface. This embedding induces group homomorphism
\[
s(\iota): \Symp_{\omega}(\Sigma,\text{rel }D^2)\to \Symp_{\omega}(\Sigma),
\]
where $ \Diff_{\omega}(\Sigma,\text{rel }D^2)$ denotes those volume preserving diffeomorphisms that are the identity in a neighborhood of the embedded disk. We shall prove below that $s(\iota)$ induces homology isomorphism in a range. 
\begin{rem}\label{rem1}
Similarly the volume preserving embedding induce a homomorphism 
\[
h(\iota): \deHam_{\omega}(\Sigma,\text{rel }D^2)\to \deHam_{\omega}(\Sigma).
\]
But the map $h(\iota)$ fails to induce an isomorphism even in the low homological degrees. To see why, consider the fibration
\[
1\to \dHam_{\omega}(\Sigma)\to \deHam_{\omega}(\Sigma)\to \text{MCG}(\Sigma)\to 1.
\]
The Serre spectral sequence implies that there is a short exact sequence
\[
 H^1(\dHam_{\omega}(\Sigma);\bQ)^{\text{MCG}(\Sigma)}\to H^2(\text{MCG}(\Sigma);\bQ)\to H^2(\deHam_{\omega}(\Sigma);\bQ).
\]
Now $\dHam_{\omega}(\Sigma)$ is a perfect group by an unpublished  result of Thurston (see \cite{MR1445290}). But for $g(\Sigma)\geq 3$ the group $H^2(\text{MCG}(\Sigma);\bZ)$ is generated by the first MMM-class $\kappa_1$ (see \cite{harer1983second}). Thus the class $\kappa_1$ is also nonzero in $H^2(\deHam_{\omega'}(\Sigma');\bQ)$. But $\Ham(\Sigma,\text{rel }D^2)$ is not perfect and Bowden observed in \cite[Theorem 7.2]{MR2826937} that for this reason $\kappa_1$ in $H^2(\deHam_{\omega}(\Sigma,\text{rel }D^2);\bQ)$ has to vanish.
\end{rem}

To prove that $s(\iota)$ induces a homology isomorphism in the same range as \Cref{thm1}, we use a modification of the disk resolution technique (\cite[Section 11.2]{randal2009resolutions}).
\begin{defn}
Let $[p]$ denote the set $\{0,1,...,p\}$. Let $ \text{\textnormal{Emb}}_{\omega}(\coprod_{[p]} D^2, \Sigma)$ denote the space of smooth volume preserving embeddings of union of $p$ disjoint  closed $2$-disks with the standard volume form into the surface $\Sigma$. We say two volume preserving  embeddings $g_1$ and $g_2$ in $ \text{\textnormal{Emb}}(\coprod_{[p]} D^2, \Sigma)$ have the same germ if there exists an open neighborhood $U\subset D^2$ around the origin so that $g_1|_{\coprod_{[p]} U}= g_2|_{\coprod_{[p]} U}$. Let 
\[E_p(\Sigma):= \text{\textnormal{Emb}}^{g, \delta}_{\omega}(\coprod_{[p]} D^2, \Sigma).
\]
be the space of germs of volume preserving embeddings with the discrete topology
\end{defn}
Note that $E_{\bullet}(\Sigma)$ is a semisimplicial set whose face maps are given by forgetting the disks (see \cite[Section 2]{randal2009resolutions} for preliminaries on semisimplicial spaces). Using isotopy extension theorem for volume preserving diffeomorphisms (see \cite[Theorem 2]{krygin1971continuation}), one can see that the group $\dDiff_{\omega}(\Sigma)$ in fact acts transitively on $E_p(\Sigma)$. 

Let us fix $e_p\in E_p(\Sigma)$ for each $p$ so that it has a representative whose  image  does not intersect our fixed embedded disk $\iota: D^2\hookrightarrow \Sigma$. We use the same notation for this representative of the germ of embeddings $e_p$. Let $\Sigma(p)$ denote the punctured surface obtained by removing the centers of the disks in $e_p$. Let $\Sigma\backslash e_p$ denote the surface obtained by removing the interior of the image of $e_p$ from $\Sigma$. 

The action of $\dSymp_{\omega}(\Sigma)$ on $e_p$ defines a map 
\begin{equation}\label{eq6}
\pi: \dSymp_{\omega}(\Sigma)\to E_p(\Sigma).
\end{equation}
 The fiber of the map $\pi$ over $e_p$ is $\dSymp_{c,\omega}(\Sigma(p))$ which is compactly supported volume preserving diffeomorphisms of $\Sigma(p)$. 
\begin{defn}\label{defn}The disk resolution for $\BdSymp_{\omega}(\Sigma)$ is defined to be the augmented semisimplicial space 

\[ \epsilon: X_{\bullet}(\Sigma):= E_{\bullet}(\Sigma)\hcoker \dSymp_{\omega}(\Sigma)\to \BdSymp_{\omega}(\Sigma),\]
whose face maps are induced by that of $E_{\bullet}(\Sigma)$.
\end{defn}
Consider the map 
\[
|\epsilon|: |X_{\bullet}(\Sigma)|\xrightarrow{}\BdSymp_{\omega}(\Sigma),
\]
induced by $\epsilon$. By \cite[Lemma 2.1]{randal2009resolutions}, the homotopy fiber of this map is the realization $|E_{\bullet}(\Sigma)|$. But as we show below $|E_{\bullet}(\Sigma)|$ is contractible. Therefore $|\epsilon|$ is a weak homotopy equivalence. 
\begin{lem}\label{claim}
The realization $|E_{\bullet}(\Sigma)|$ is weakly contractible.
\end{lem}
\begin{proof}
Let $ f: S^k\to |E_{\bullet}(\Sigma)|$ represents an element in the $k$-th homotopy group of $|E_{\bullet}(\Sigma)|$. Since $|E_{\bullet}(\Sigma)|$ is a CW-complex and $S^k$ is compact, the map $f$ hits finitely many simplices of $|E_{\bullet}(\Sigma)|$. Hence, there exists a point   ${\bf p}$ in $\Sigma$ and an embedded disk $e(D^2)$ around it such that as an element of $E_{0}(\Sigma)$ is disjoint from the centers of the germs of embedded disks in the image of $f$. Therefore, we have $f(S^k)\subset |E_{\bullet}(\Sigma\backslash e(D^2))|$. Adding the germ of $e$ at ${\bf p}$ to the list of germs of embeddings of disks in $\Sigma\backslash e(D^2)$ gives a semisimplicial null-homotopy for the inclusion $E_{\bullet}(\Sigma\backslash e(D^2))\hookrightarrow E_{\bullet}(\Sigma)$. Hence, $f(S^k)$ can be coned off inside $|E_{\bullet}(\Sigma)|$.
\end{proof}
Given that $|\epsilon|$ induces a weak homotopy equivalence, the spectral sequence associated to the skeleton filtration of the realization $|X_{\bullet}(\Sigma)|$ takes the form
\[
E^1_{p,q}(\Sigma)=H_{q}(X_p(\Sigma);\bZ)\Longrightarrow H_{p+q}(\BdSymp_{\omega}(\Sigma);\bZ).
\]

\begin{proof}[Proof of \Cref{cappingoffsymp}] We can similarly define a disk resolution $X_{\bullet}(\Sigma,\text{rel }D^2)$ for the open surface $\Sigma\backslash D^2$. The stabilization map induce a semisimplicial map between augmented semisimplicial spaces
\[
X_{\bullet}(\Sigma,\text{rel }D^2)\to X_{\bullet}(\Sigma).
\]
Hence we obtain a comparison map between the associated spectral sequences
\[
\begin{tikzcd}
H_q(X_{p}(\Sigma,\text{rel }D^2))\arrow{r}\arrow[Rightarrow]{d}&H_q(X_{p}(\Sigma))\arrow[Rightarrow]{d}\\H_{p+q}(|X_{\bullet}(\Sigma,\text{rel }D^2)|)\arrow{r}\arrow["\cong"]{d}& H_{p+q}(|X_{\bullet}(\Sigma)|)\arrow["\cong"]{d}\\ H_{p+q}(\BdSymp_{\omega}(\Sigma,\text{rel }D^2))\arrow["\iota_*"]{r}& H_{p+q}(\BdSymp_{\omega}(\Sigma)).
\end{tikzcd}
\]
The action in \Cref{defn}, yields a sequence of fibrations
\[
\dSymp_{c,\omega}(\Sigma(p))\to \dSymp_{\omega}(\Sigma)\to E_p(\Sigma)\to X_p(\Sigma)\to \BdSymp_{\omega}(\Sigma).
\]
 Now by  Shapiro's lemma (which says that for a subgroup $H<G$, the homotopy quotient $(G/H)\hcoker G$ is weakly equivalent to $\mathrm{B}H$), we have 
\[
X_p(\Sigma)\simeq \BdSymp_{c,\omega}(\Sigma(p)).
\]
\[
X_p(\Sigma, \text{rel }D^2)\simeq \BdSymp_{c,\omega}(\Sigma(p)\backslash D^2).
\]
Now note that similar to \cite[Corollary 2.3]{nariman2014homologicalstability}, the inclusion $\dSymp_{\omega}(\Sigma\backslash e_p,\partial)\hookrightarrow\dSymp_{c,\omega}(\Sigma(p))$ induces a homology isomorphism  where $\Sigma\backslash e_p$ has $p+1$ boundary components. Therefore, in the commutative diagram
\[
 \begin{tikzpicture}[node distance=6.3cm, auto]
  \node (A) {$E^1_{p,q}(\Sigma,\text{rel }D^2)=H_q(\BdSymp_{c,\omega}(\Sigma(p)\backslash D^2))$};
  \node (B) [right of=A]{$H_q(\BdSymp_{c,\omega}(\Sigma(p)))=E^1_{p,q}(\Sigma)$};
  \node (C) [below of=A, node distance=2.2cm]{$H_q(\BdSymp_{\omega}(\Sigma\backslash e_p\cup D^2,\partial))$};
      \node (D)[right of=C]  {$H_q(\BdSymp_{\omega}(\Sigma\backslash e_p,\partial))
,$};
  \draw [->] (A) to node {$$} (B);
  \draw [->] (B) to node {$\cong$} (D);
  \draw [-> ] (A) to node {$\cong$} (C);
    \draw [<-] (D) to node {$$} (C);
 \end{tikzpicture}
 \]
the bottom map is an isomorphism for $q\leq (2g(\Sigma)-2)/3$  by \Cref{thm1} for surfaces with boundary. Therefore the top horizontal map between $E^1$-pages is an isomorphism in the same range which implies that the map
$$H_{p+q}(\BdSymp_{\omega}(\Sigma,\text{rel }D^2))\to H_{p+q}(\BdSymp_{\omega}(\Sigma))$$
is an isomorphism in the same range.
\end{proof}
\subsection{The stable homology} From now on, for brevity, we shall drop the symplectic form $\omega$ from the subscripts. As we saw in \Cref{rem1}, the extended Hamiltonian groups do not exhibit homological stability when we cap off the last boundary component. Nonetheless one can describe the group homology of $\deHam(\Sigma)$ for a closed surface $\Sigma$ as in \Cref{thm3'}. Before proving \Cref{thm3'} which is the main goal of this section, let us recall  from \Cref{tang} that $\BdSymp(\Sigma,\partial)$ is homology isomorphic to the base point component of $\mathcal{M}^{\theta}(\Sigma,\partial)$ and similarly we showed that $\BdeHam(\Sigma,\partial)$ is homology isomorphic to $\mathcal{M}^{\beta}(\Sigma,\partial)$. 

Therefore similar to \cite[Theorem 2.2]{nariman2015stable},  \Cref{thm3} is implied by the homological stability for $\mathcal{M}^{\theta}(\Sigma,\partial)$ and $\mathcal{M}^{\beta}(\Sigma,\partial)$, and the main theorem of Galatius-Madsen-Tillmann-Weiss in \cite{galatius2009homotopy}. Hence by a standard argument, one obtains maps arising from the Pontryagin-Thom construction
\[
\BdSymp(\Sigma,\partial)\to \Omega_{\bullet}^{\infty}\text{MT}\theta,
\]
\[
\BdeHam(\Sigma,\partial)\to \Omega_{\bullet}^{\infty}\text{MT}\beta,
\]
that induce  isomorphisms on homology in degrees less than or equal to $(2g(\Sigma)-2)/3$ and surjections in degrees less than $(2g(\Sigma)+1)/3$. If the surface $\Sigma$ is closed, the stable homology of $\BdSymp(\Sigma)$  also coincides with that of $\Omega_{\bullet}^{\infty}\text{MT}\theta$, but  the situation is different for $\BdeHam(\Sigma)$. One should note that the moduli space of $\beta$-structures exhibit homological stability even when one caps off the last boundary component. But as we shall explain below the map 
\[
\BdeHam(\Sigma)\to \mathcal{M}^{\beta}(\Sigma),
\]
is no longer a homology isomorphism even in the stable range.  We need to mod out $\mathcal{M}^{\beta}(\Sigma)$ by a certain subgroup of the homotopy automorphism group of the tangential structure $\beta$. Let us first recall the definition of the homotopy automorphism group of a map.
\begin{defn}
Let $\pi: E\to B$ be a fibration. The topological monoid $\text{hAut}(\pi)$ is the space of maps $f:E\to E$ which are weak homotopy equivalences and satisfy $\pi\circ f=f$. The monoid structure is induced by the composition.
\end{defn} 

We are interested in $\text{hAut}(\beta)$ for the tangential structure $\beta:\widetilde{\BH}\to \mathrm{BSL}_2(\bR)$. Note that $\text{hAut}(\beta)$ acts on $\mathcal{M}^{\beta}(\Sigma)$ and on the spectrum $\text{MT}\beta$. To prove \Cref{thm3'}, we first describe an action of the topological abelian group $\mathrm{B}\bR^{\delta}$ on the spectrum $\text{MT}\beta$ by realizing it as a submonoid of $\text{hAut}(\beta)$. 

Let $E(\mathrm{B}\bR^{\delta})$ denote the universal $\mathrm{B}\bR^{\delta}$-principal bundle over $K(\bR,2)\simeq \mathrm{B}(\mathrm{B}\bR^{\delta})$. Consider the model for $\widetilde{\BH}$ obtained by the  homotopy pullback diagram
\begin{equation}\label{Z}
  \begin{gathered}
 \begin{tikzpicture}[node distance=1.5cm, auto]
  \node (C) {$\widetilde{\BH}$};
  \node (D) [below of=C, node distance= 1.5cm]{$\BH$};
  \node (E) [right of=C, node distance=4.2cm]{$\mathrm{BSL}_2(\bR)\times E(\mathrm{B}\bR^{\delta}) $};
    \node (G) [right of=E, node distance=3.2cm]{$\mathrm{BSL}_2(\bR)$};
  \node (F) [below of=E]{$\mathrm{BSL}_2(\bR)\times \mathrm{B}(\mathrm{B}\bR^{\delta}),$};
  \draw [->] (C) to node {$$} (D);
  \draw [->] (C) to node {$$} (E);
  \draw [<-] (F) to node {$$} (E);
    \draw [->] (E) to node {$\simeq$} (G);
  \draw [->] (D) to node {$(\theta,-\circ (e+v))$} (F);
 \end{tikzpicture}
   \end{gathered}
 \end{equation}
where the composition of the top horizontal maps is $\beta$. Using this model $\widetilde{\BH}$ admits  an action of $\mathrm{B}\bR^{\delta}$ as the principal $\mathrm{B}\bR^{\delta}-$bundle over  $\BH$. Hence, from the diagram \ref{Z}, we obtain that this $\mathrm{B}\bR^{\delta}$-action fixes the map $\beta$. Therefore, $\mathrm{B}\bR^{\delta}$ is a submonoid of $\text{hAut}(\beta)$. So it also acts on the Thom spectrum $\text{MT}\beta$. We want to prove that for a closed surface $\Sigma$, there is a map 
\[
\BdeHam(\Sigma)\to  \mathrm{B}\bR^{\delta}\hker\Omega^{\infty}\text{MT}\beta,
\]
that induces homology isomorphism in the stable range onto the connected component that it hits.
\begin{proof}[Proof of \Cref{thm3'}] Recall from (\ref{fib1}) that we have the fibration sequence
\[
\text{Bun}(T\Sigma,{\beta}^*\gamma)\xrightarrow{g} \text{Bun}_{0}(T\Sigma,{\theta}^*\gamma)\xrightarrow{-\circ (e+v)} \text{Map}_{0}(\Sigma,K(\bR,2)).
\]
Since the group $\mathrm{B}\bR^{\delta}$ is a subgroup of $\text{hAut}(\beta)$, it also acts on $\text{Bun}(T\Sigma,{\beta}^*\gamma)$. Given the model for $\widetilde{\BH}$ in the diagram \ref{Z}, the map $g$ factors through the homotopy quotient of this action
 \[
 \begin{tikzpicture}[node distance=2.3cm, auto]
  \node (A) {$\text{Bun}(T\Sigma,{\beta}^*\gamma)$};
  \node (B) [right of=A, below of=A, node distance=1.6cm]{$ \mathrm{B}\bR^{\delta}\hker\text{Bun}(T\Sigma,{\beta}^*\gamma)$};
  \node (C) [right of=A, node distance=4.3cm]{$ \text{Bun}_{0}(T\Sigma,{\theta}^*\gamma).$};
  \draw [->] (A) to node {$g$} (C);
  \draw [->] (B) to node {$$} (C);
  \draw [->] (A) to node {} (B);
 \end{tikzpicture}
\]

On the other hand, unlike the case of surfaces with nonempty boundary, the map from the mapping space $\text{Map}_{0}(\Sigma,K(\bR,2))$ to $\mathrm{B}{H^1_{\bR}}^{\delta}$ is no longer a weak equivalence. In fact, we have the fibration sequence
\[
K(\bR,2)\to \text{Map}_{0}(\Sigma,K(\bR,2))\xrightarrow{p} \mathrm{B}{H^1_{\bR}}^{\delta}.
\]
Let $X$ denote the homotopy fiber of the map $p\circ (-\circ (e+v))$, then $X$ fits into the following diagram

\begin{equation}\label{S}
  \begin{gathered}
 \begin{tikzpicture}[node distance=2cm, auto]
  \node (A) {$ \text{Bun}(T\Sigma,{\beta}^*\gamma)$};
  \node (B) [right of=A, node distance=3cm]{$X$};
  \node (C) [right of=B, node distance=4.2cm]{$\mathrm{B}(\mathrm{B}\bR^{\delta})$};
  \node (D) [below of=A, node distance=1.5cm]{$  \text{Bun}(T\Sigma,{\beta}^*\gamma)$};
  \node (E) [right of=D, node distance=3cm]{$ \text{Bun}_{0}(T\Sigma,{\theta}^*\gamma)$};
  \node (F) [right of=E, node distance=4.2cm]{$ \text{Map}_{0}(\Sigma,K(\bR,2))$};
    \node (G) [below of=D, node distance=1.5cm]{$ *$};
  \node (H) [right of=G, node distance=3cm]{$ \mathrm{B}{H^1_{\bR}}^{\delta}$};
  \node (I) [right of=H, node distance=4.2cm]{$\mathrm{B}{H^1_{\bR}}^{\delta},$};
  \draw [->] (A) to node {$$} (B);
  \draw [->] (B) to node {$$} (C);
    \draw [->] (C) to node {$$} (F);
    \draw [->] (B) to node {$$} (E);
  \draw [->] (A) to node {$\cong$} (D);
  \draw [->] (D) to node {$$} (E);
  \draw [->] (E) to node {$-\circ (e+v)$} (F);
    \draw [->] (E) to node {} (H);
  \draw [->] (D) to node {$$} (G);
  \draw [->] (G) to node {$$} (H);
    \draw [->] (H) to node {$\cong$} (I);
        \draw [->] (F) to node {$p$} (I);
 \end{tikzpicture}
   \end{gathered}
 \end{equation}
where every horizontal and vertical line is a fibration. Recall that for a group $G$ and a topological space $Y$, the group $G$ acts on a model for the homotopy fiber of a map $f: Y\to  \mathrm{B}G$.  And the total space $Y$ is homotopy equivalent to the homotopy quotient $G\hker\text{hofib}(f)$.
Given that $\mathrm{B}\bR^{\delta}$ acts on the space $\text{Bun}(T\Sigma,{\beta}^*\gamma)$ via homotopy automorphisms,  from the top horizontal fibration, we deduce that $X\simeq\mathrm{B}\bR^{\delta}\hker\text{Bun}(T\Sigma,{\beta}^*\gamma)$.

McDuff's theorem \ref{mcduff1} for a closed surface $\Sigma$ implies that the map  $\tilde{f_{\Sigma}}$ in the diagram 
\begin{equation}\label{W}
  \begin{gathered}
 \begin{tikzpicture}[node distance=2cm, auto]
  \node (A) {$\BdHam(\Sigma)$};
  \node (B) [below of=A]{$X$};
  \node (C) [right of=A, node distance=3cm]{$\BdSymp_0(\Sigma)$};
  \node (D) [right of=B, node distance=3cm]{$  \text{Bun}_{0}(T\Sigma,{\theta}^*\gamma)$};
  \node (E) [right of=C, node distance=4.7cm]{$\mathrm{B}{H^1_{\bR}}^{\delta}$};
  \node (F) [below of=E]{$\mathrm{B}{H^1_{\bR}}^{\delta},$};
  \draw [->] (A) to node {$$} (C);
  \draw [->] (B) to node {$$} (D);
  \draw [->] (A) to node {$\tilde{f_{\Sigma}}$} (B);
  \draw [->] (C) to node {$f_{\Sigma}$} (D);
  \draw [->] (C) to node {BFlux} (E);
  \draw [->] (E) to node {$\cong$} (F);
  \draw [->] (D) to node {$p\circ (-\circ (e+v))$} (F);
 \end{tikzpicture}
   \end{gathered}
 \end{equation}
 induces a homology isomorphism. As we discussed in the proof of \Cref{thm2}, the group $\Symp(\Sigma)$ acts on a model for $\BdHam(\Sigma)$. Also recall that $\Symp(\Sigma)$ acts on $\text{Bun}(T\Sigma,{\beta}^*\gamma)$ by acting on the tangent bundle $T\Sigma$ and  the topological group $\mathrm{B}\bR^{\delta}$ acts on  $\text{Bun}(T\Sigma,{\beta}^*\gamma)$  via homotopy automorphisms of $\beta$. Therefore, the action of $\mathrm{B}\bR^{\delta}$ and the action of $\Symp(\Sigma)$ on $\text{Bun}(T\Sigma,{\beta}^*\gamma)$ commute which implies that $\Symp(\Sigma)$ acts on $\mathrm{B}\bR^{\delta}\hker\text{Bun}(T\Sigma,{\beta}^*\gamma)$ as  a model for $X$ still by acting on $T\Sigma$.  Hence, the construction for $\tilde{f}_{\Sigma}$ as in \cite[Section 5.1]{nariman2014homologicalstability} makes it $\Symp(\Sigma)-$equivariant. So we have a map
 \[
 \mathrm{B}\deHam(\Sigma)\simeq \BdHam(\Sigma)\hcoker \Symp(\Sigma)\longrightarrow \mathrm{B}\bR^{\delta}\hker\text{Bun}(T\Sigma,{\beta}^*\gamma)\hcoker \Symp(\Sigma),
 \]
that induces a homology isomorphism. Now from \cite{galatius2009homotopy}, we know that the stable homology of $\mathcal{M}^{\beta}(\Sigma)=\text{Bun}(T\Sigma,{\beta}^*\gamma)\hcoker \Symp(\Sigma)$ coincides with that of a connected component of $\Omega^{\infty}\text{MT}\beta$. Hence, we obtain a map  
\[
\BdeHam(\Sigma)\to \mathrm{B}\bR^{\delta}\hker \Omega^{\infty}\text{MT}\beta, 
\]
that induces a homology isomorphism in the stable range onto the connected component that it hits.
 \end{proof}
 \begin{proof}[Proof of \Cref{cor1}] Recall that we want to show that for every prime $p$, the map induced by capping off the last boundary component
\[
H_*(\mathrm{B}\deHam(\Sigma,\partial);\bF_p)\to H_*(\mathrm{B}\deHam(\Sigma);\bF_p),
\]
is an isomorphism on homology  in the stable range.  It is enough to show the map 
\[
\Omega^{\infty}\text{MT}\beta\to \mathrm{B}\bR^{\delta}\hker \Omega^{\infty}\text{MT}\beta,
\]
induces homology isomorphism with $\bF_p$-coefficients. Using a $\bQ$-basis for $\bR$ and the Kunneth formula, one can show that 
\begin{equation}\label{eq7}\begin{gathered}
H_k(K(\bR,2);\bZ)=
\begin{cases}
\bZ & k=0\\ S _{\bQ}^r(\bR) & k=2r\\
0 & \text{otherwise}
\end{cases}
\end{gathered}
\end{equation}
where $S _{\bQ}^r(\bR)$ is the $r$-th symmetric power of $\bR$ as a $\bQ$-vector space. Since $S _{\bQ}^r(\bR)$ is a uniquely divisible abelian group, the universal coefficient theorem implies that $K(\bR,2)$ has the $\bF_p$-homology of the point. Therefore, the Serre spectral sequence for the fibration 
\[
\Omega^{\infty}\text{MT}\beta\to \mathrm{B}\bR^{\delta}\hker \Omega^{\infty}\text{MT}\beta\to K(\bR,2),
\]
degenerates and we obtain the desired isomorphism 
\[
H_*(\Omega^{\infty}\text{MT}\beta;\bF_p)\xrightarrow{\cong}H_*(\mathrm{B}\bR^{\delta}\hker \Omega^{\infty}\text{MT}\beta;\bF_p).
\]
 \end{proof}
 The above proof shows that although capping off the last boundary component for $\deHam(\Sigma)$ does not exhibit homological stability with rational coefficients, it does with finite coefficients. Thus for closed surfaces   $\Sigma$ and $\Sigma'$, one can use \Cref{cor1} to find a zig-zag of isomorphisms between $H_*(\mathrm{B}\deHam(\Sigma);\bF_p)$ and $H_*(\mathrm{B}\deHam(\Sigma');\bF_p)$ in the stable range,  even if there is no direct map between $\deHam(\Sigma)$ and $\deHam(\Sigma')$. 
 
Moreover, we show below that the rational group homology of $\deHam(\Sigma)$ and $\deHam(\Sigma')$ are isomorphic in the stable range via a different zig-zag of maps.
 \begin{proof}[Proof of \Cref{cor2'}] First note that although different connected components of $\Omega^{\infty}\text{MT}\beta$ have the same homotopy type, there is no reason for different connected components of $\mathrm{B}\bR^{\delta}\hker \Omega^{\infty}\text{MT}\beta$ to be homotopy equivalent. Recall that the group of connected components of $\Omega^{\infty}\text{MT}\beta$ maps to the index $2$ subgroup of $H_2(\widetilde{\BH};\bZ)$ as follows
 \[
 H_0(\Omega^{\infty}\text{MT}\beta;\bZ)= \pi_0(\text{MT}\beta;\bZ)\rightarrow H_2(\widetilde{\BH};\bZ)\xrightarrow{\cong}\bZ,
 \]
where the last isomorphism is given by the Euler class of the tangential structure $\beta:  \widetilde{\BH}\to \mathrm{BSL}_2(\bR)$. Therefore  the group of connected components of $\Omega^{\infty}\text{MT}\beta$ is isomorphic to $\bZ$ by half of the Euler class. Let $\Omega^{\infty}_n\text{MT}\beta$ denote the connected component corresponding to $n$ in the above isomorphism. Therefore, for a surface $F$, similar to \cite[Section 2.4]{MR1856399} we have a map
 \[
 \BdeHam(F,\partial)\to \Omega^{\infty}\text{MT}\beta,
 \]
that hits the $\frac{\chi(F)}{2}$-th connected component.  Given that the action of $\mathrm{B}\bR^{\delta}$ preserves the connected components, by \Cref{thm3'} we have maps
 \[
\BdeHam(\Sigma) \to \mathrm{B}\bR^{\delta}\hker \Omega^{\infty}_{\chi(\Sigma)/2}\text{MT}\beta,
 \]
 \[
\BdeHam(\Sigma') \to \mathrm{B}\bR^{\delta}\hker \Omega^{\infty}_{\chi(\Sigma')/2}\text{MT}\beta,
 \] 
 that induce homology isomorphisms in the stable range. 
 
 \noindent{\it Claim: }For every $n$ and $k$, there exists a map
 \[
\phi_k: \Omega^{\infty}_{n}\text{MT}\beta\to \Omega^{\infty}_{kn}\text{MT}\beta,
 \]
 that commutes with the action of $ \mathrm{B}\bR^{\delta}$ and induces an isomorphism on homology with rational coefficients. 
 
 To construct the map $\phi_k$, write the infinite loop space $\Omega^{\infty}\text{MT}\beta$ as a loop space $\Omega Y$. Recall that $\pi_0(\Omega Y)=\bZ$. By traversing each loop $k$ times, one obtains a map
 \[
 \phi_k: \Omega Y\to \Omega Y,
 \]
 that induces multiplication by $k$ on $\pi_0(\Omega Y)$. This map obviously commutes with the action of $\mathrm{B}\bR^{\delta}$ and is invertible after rationalization.  Therefore it induces an isomorphism  on homology with rational coefficients.

 Let us explain how to use the claim to finish the proof. We assume that neither $\chi(\Sigma)$ nor $\chi(\Sigma')$ is zero. Consider the following diagram of spaces
\begin{equation} \begin{gathered}
 \begin{tikzpicture}[node distance=2.cm, auto]
  \node (C) {$ \Omega^{\infty}_{\chi(\Sigma)/2}\text{MT}\beta$};
  \node (D) [below of=C, node distance= 2.cm]{$\Omega^{\infty}_{\chi(\Sigma)\chi(\Sigma')/4}\text{MT}\beta$};
  \node (E) [right of=C, node distance=4.2cm]{$ \Omega^{\infty}_{\chi(\Sigma')/2}\text{MT}\beta $};
  \node (F) [below of=E]{$ \Omega^{\infty}_{\chi(\Sigma)\chi(\Sigma')/4}\text{MT}\beta.$};
  \draw [->] (C) to node {$\phi_{\chi(\Sigma')/2}$} (D);
  \draw [<-] (F) to node {$\phi_{\chi(\Sigma')/2}$} (E);
  \draw [->] (D) to node {$=$} (F);
 \end{tikzpicture}
   \end{gathered}
 \end{equation}
 The vertical maps are isomorphisms on rational homology.  Hence after taking homotopy quotient by $\mathrm{B}\bR^{\delta}$, we obtain the desired zig-zag of maps that induce isomorphisms on rational homology. 
  \end{proof}
\section{Characteristic classes of symplectic flat surface bundles}\label{sec3} In this section, we show how one can use \Cref{thm3} to prove non-triviality or vanishing of characteristic classes for flat surface bundles whose holonomy groups are area preserving. We call such surface bundles symplectic flat surface bundles. For a surface $\Sigma$, the invariants of a flat $\Sigma$-bundle with a transverse volume form live in $H^*(\BdSymp(\Sigma);\bZ)$. If the holonomy has vanishing extended flux, then the invariants come from classes in $H^*(\BdeHam(\Sigma);\bZ)$. To construct characteristic classes, we consider  the universal symplectic flat surface bundle 
\[
\begin{tikzpicture}[node distance=1.4cm, auto]
  \node (C) {$ \Sigma$};
  \node (E) [right of=C, node distance=2cm]{$ \Sigma\hcoker \dSymp(\Sigma) $};
  \node (F) [below of=E]{$ \BdSymp(\Sigma).$};
  \draw [->] (C) to node {$$} (E);
  \draw [<-] (F) to node {$\pi$} (E);
 \end{tikzpicture}
\]
 Here are certain characteristic classes that one can define for such symplectic flat surface bundle:
\begin{itemize}
\item {\bf MMM-classes.} Let $T_{\pi}$ be the vertical tangent bundle which is $2$-plane bundle on the total space, tangent to the fibers. Let $e(T_{\pi})$ be the Euler class of this bundle. The MMM-classes are defined to be
\[
\kappa_i=\pi_!(e(T_{\pi})^{i+1})\in H^{2i}(\BdSymp(\Sigma);\bZ).
\]
In other words, one can forget that the bundle is foliated and just consider its invariants as a surface bundle. Such classes  come from the cohomology of the mapping class group of $\Sigma$. 
\item {\bf Characteristic classes of foliations.} There are certain characteristic classes associated to foliations with transverse volume forms that in our case live in $H^*(\BH;\bR)$ (see \cite{MR712266} and \cite{MR0312531} for different constructions of such classes). A symplectic flat surface bundle, in particular provides a codimension $2$ foliation with a transverse volume form on the total space. The pushforward of such classes live in $H^*(\BdSymp(\Sigma);\bR)$. 
\item {\bf Kotschick-Morita classes.} Kotschick and Morita used the extended flux as a twisted class to build interesting invariants of symplectic flat surface bundles (see \cite{kotschick2004characteristic} for details). Their classes live in the cohomology group $H^*(\BdSymp(\Sigma); S^k_{\bQ}(S^2\bR))$. 
\end{itemize}
One of the consequence of  \Cref{thm3}, as we shall see below, is in fact Kotschick-Morita's classes are induced from characteristic classes of foliations. 

For every $n$, let us denote the following composition by $e_n$
\begin{equation}\label{eq9}\begin{gathered}
e_n: H_n(\BdSymp(\Sigma);\bZ)\to H_n(\Omega^{\infty}\text{MT}\theta;\bZ)\to H_{n+2}(\BH;\bZ),
\end{gathered}\end{equation}
where the first map is induced by a Pontryagin-Thom construction (see \cite[Section 2.2]{nariman2015stable} for a description of such a map) and the second map is given by the Thom isomorphism. For homology with rational coefficients, one can geometrically describe this map as follows. Recall that from a theorem of Thom, every class in $c\in H_n(\BdSymp(\Sigma);\bQ)$ can be represented by $\Sigma\to E_c\to M_c$ which is a symplectic flat $\Sigma$-bundle over an $n$-manifold $M_c$. By definition, this bundle gives rise to a codimension $2$ foliation on $E_c$ with a transverse volume form. One  can easily check that the map that associates the class $[E_c]\in H_{n+2}(\BH,\bQ)$ to the class $c$ gives a well-defined map
\[
H_n(\BdSymp(\Sigma);\bQ)\to H_{n+2}(\BH,\bQ).
\]
Let $\Sigma$ be a surface with boundary, we can define a similar map for $\BdSymp(\Sigma,\partial)$ and $\BdeHam(\Sigma,\partial)$ and in the case of extended Hamiltonians, we obtain a map
\[
h_n: H_n(\BdeHam(\Sigma,\partial);\bZ)\to H_{n+2}(\widetilde{\BH},\bZ).
\]
\begin{prop}\label{prop1}
For a surface $\Sigma$ with boundary, the maps $e_n$ and $h_n$ are rationally surjective for $0<n\leq 2g(\Sigma)/3$.
\end{prop} 
\begin{proof}
We prove surjectivity for $e_n$ and the proof for $h_n$ is similar. From \Cref{thm3}, we know that  for $n\leq 2g(\Sigma)/3$, there is a surjective map
\[
\begin{tikzpicture}[node distance=4cm]
  \node (C) {$H_n(\BdSymp(\Sigma,\partial);\bQ)$};
  \node (D) [right of=C]{$H_n(\Omega^{\infty}_{\bullet}\text{MT}\theta;\bQ),$};
\draw[->>] (C) to node {} (D);
\end{tikzpicture}
\]
where $\Omega^{\infty}_{\bullet}\text{MT}\theta$ means the base point component of $\Omega^{\infty}\text{MT}\theta$. The map induced by the suspension map $\Omega^{\infty}\text{MT}\theta\to \text{MT}\theta$ followed by Thom isomorphism gives the map
\[
H_n(\Omega^{\infty}_{\bullet}\text{MT}\theta;\bQ)\to H_n(\text{MT}\theta;\bQ)\xrightarrow{\cong}H_{n+2}(\BH;\bQ).
\]
Hence, it is enough to prove the above map is surjective. Consider the commutative diagram
\[
\begin{tikzcd}
\pi_n(\Omega^{\infty}_{\bullet}\text{MT}\theta)\otimes \bQ\arrow{r}\arrow{d}&\pi_n(\text{MT}\theta)\otimes\bQ\arrow{d}\\ H_n(\Omega^{\infty}_{\bullet}\text{MT}\theta;\bQ)\arrow[two heads]{r}&H_n(\text{MT}\theta;\bQ),
\end{tikzcd}
\] 
where the horizontal maps are induced by the suspension map and the vertical maps are induced by the Hurewicz map.  The top horizontal map is an isomorphism by the definition of the homotopy groups of a spectra  and the right vertical map is also an isomorphism because of the rational Hurewicz theorem. Therefore, the bottom horizontal map, is surjective.
\end{proof}
Hence,  nontrivial classes in $H^*(\BH;\bQ)$ and $H^*(\widetilde{\BH};\bQ)$ give rise to  nontrivial classes in $H^*(\BdSymp(\Sigma,\partial);\bQ)$ and $H^*(\BdeHam(\Sigma,\partial);\bQ)$ respectively. We investigate these two cases separately.
\subsection{Characteristic classes of flat surface bundles whose holonomy groups lie in extended Hamiltonian} Recall from the introduction that we have a short exact sequence
\begin{equation}\label{eq4} \begin{gathered}
1\to \Ham(\Sigma,\partial)\to \widetilde{\Ham}(\Sigma,\partial)\to \text{MCG}(\Sigma,\partial)\to 1.
\end{gathered}\end{equation}
There is a surjective homomorphism called Calabi homomorphism $$ \text{Cal}: \dHam(\Sigma,\partial)\to \bR.$$ Banyaga \cite{MR490874} proved that the kernel of this homomorphism is perfect. Therefore, we have $
H_1(\dHam(\Sigma,\partial);\bZ)\cong\bR.$

As Bowden observed $\text{Cal}$ lives in $H^1(\dHam(\Sigma,\partial);\bR)^{\text{MCG}(\Sigma,\partial)}$. He in \cite[Theorem 7.2]{MR2826937} proved that in the cohomology Hochschild-Serre spectral sequence for the short exact sequence \ref{eq4}, the differential 
\[
E_2^{0,1}=H^1(\dHam(\Sigma,\partial);\bR)^{\text{MCG}(\Sigma,\partial)}\xrightarrow{d^2} E_2^{2,0}=H^2(\text{MCG}(\Sigma,\partial);\bR)\cong \bR,
\]
is nontrivial by showing that $d^2(\text{Cal})$ is nonzero. Hence, dually in the homology Hochschild-Serre spectral sequence with rational coefficients, we obtain a map
\[
E_{2,0}^2=H_2(\text{MCG}(\Sigma,\partial);\bQ)\cong \bQ\xrightarrow{d_2} E^2_{0,1}=H_1(\dHam(\Sigma,\partial);\bQ)\cong \bR,
\]
that is injective. Since the mapping class group $\text{MCG}(\Sigma,\partial)$ is perfect (\cite{powell1978two} proves that the first homology of the mapping class group of a closed surface of genus larger than two is trivial and the Harer homological stability \cite{harer1985stability} implies that the first homology is stable as we cap off the boundary components if the genus is larger than two) for $g(\Sigma)\geq 3$, from the homology Hochschild-Serre spectral sequence, we deduce $$H_1(\BdeHam(\Sigma,\partial);\bQ)\cong \bR/\bQ$$ if the genus is larger than $2$.
\begin{prop}\label{prop1}
For $k\leq 2g(\Sigma)/3$ and $g\geq 3$, there is a surjective map
\[
\begin{tikzcd}
H_k(\BdeHam(\Sigma,\partial);\bQ)\arrow[two heads]{r} & \bigwedge\nolimits^k_{\bQ}(\bR/\bQ).
\end{tikzcd}
\]
\end{prop}
\begin{proof}
Again by \Cref{thm2}, we know that in the same range, there is a surjective map
\[
\begin{tikzcd}
H_k(\BdeHam(\Sigma,\partial);\bQ)\arrow[two heads]{r} & H_k(\Omega^{\infty}_{\bullet}\text{MT}\beta;\bQ),
\end{tikzcd}
\]
where $\Omega^{\infty}_{\bullet}\text{MT}\beta$ denotes the base point component of $\Omega^{\infty}\text{MT}\beta$. On the other hand, $H_*(\Omega^{\infty}_{\bullet}\text{MT}\beta;\bQ)$ is a Hopf algebra over $\bQ$ and since $H_1(\Omega^{\infty}_{\bullet}\text{MT}\beta;\bQ)\cong \bR/\bQ$ consists of primitive elements, we have a surjective map
\[
\begin{tikzcd}
H_k(\Omega^{\infty}_{\bullet}\text{MT}\beta;\bQ)\arrow[two heads]{r} & \bigwedge\nolimits^k_{\bQ}(\bR/\bQ),
\end{tikzcd}
\]
where $\bigwedge\nolimits^k_{\bQ}(\bR/\bQ)$ is the $k$-th exterior power of $\bR/\bQ$ as a vector space over $\bQ$.\end{proof}
\begin{rem}
For a closed surface, the situation is different because Banyaga's theorem in this case implies that $\dHam(\Sigma)$ is perfect. Therefore for $g(\Sigma)\geq 3$, the group $\deHam(\Sigma)$ is also perfect.
\end{rem}
\begin{proof}[Proof of \Cref{thm4}] We want to show that in the stable range, all the MMM-classes $\kappa_i\in H^{2i}(\BdeHam(\Sigma,\partial);\bR)$ vanish. By \Cref{thm2}, it is enough to show that $\kappa_i\in H^{2i}(\Omega^{\infty}_{\bullet}\text{MT}\beta;\bR)$ vanishes. Let us first recall how the class $\kappa_i$ is defined as a class in $H^{2i}(\Omega^{\infty}_{\bullet}\text{MT}\beta)$. The tangential structure $\beta$ is a map
\[
\beta: \widetilde{\BH}\to \mathrm{BSL}_2(\bR).
\]
Let $e\in H^2(\mathrm{BSL}_2(\bR);\bR)$ be the Euler class. The class $\kappa_i$ is given by the composition of the following maps
\[
H_{2i}(\Omega^{\infty}_{\bullet}\text{MT}\beta;\bR)\xrightarrow{\sigma_*}H_{2i}(\text{MT}\beta;\bR)\xrightarrow{\text{Thom iso}}H_{2i+2}(\widetilde{\BH};\bR)\xrightarrow{\beta^*(e^{i+1})} \bR,
\]
where the first map is induced by the suspension map. Hence, to prove the theorem it is enough to show that $\beta^*(e^{i+1})$ vanishes in $H^{2i+2}(\widetilde{\BH};\bR)$ for $i>0$. Recall we have commutative diagram of tangential structures
\[
\begin{tikzpicture}[node distance=1.5cm, auto]
  \node (C) {$\widetilde{\BH}$};
  \node (D) [right of=C]{$\BH$};
    \node (E) [below of=D]{$\mathrm{BSL}_2(\bR).$};
      \draw[->] (C) to node {$\alpha$} (D);
      \draw[<-] (E) to node {$\beta$} (C);
      \draw[->] (D) to node {$\theta$} (E);
\end{tikzpicture}
\]
With abuse of notation, we already denoted the pullback of the Euler class $\theta^*(e)\in H^2(\BH;\bR)$ by $e$. Since, $\widetilde{\BH}$ is the homotopy fiber of the map $$\BH\xrightarrow{e+v} K(\bR,2),$$ the class $\beta^*(e)$ is equal to $-\alpha^*(v)$ in $H^2(\widetilde{\BH};\bR)$. Note that since $v$ is the universal transverse volume form, we have $v^2=0$ in $H^2(\BH;\bR)$. Therefore $\beta^*(e^2)=0$ in $H^{4}(\widetilde{\BH};\bR)$ which concludes the proof.
\end{proof}
\begin{rem}
For a closed surface $\Sigma$, Bowden in his thesis (\cite{bowden2010two}) observed that the Bott vanishing theorem for foliations with transverse volume form which in this case says $e^2v=0\in H^6(\BH;\bR)$, implies $\kappa_i$ for $i>1$ and $\kappa_1^2$ vanish in $H^*(\BdeHam(\Sigma);\bR)$. It is an immediate consequence of the perfectness of $\dHam(\Sigma)$  that $\kappa_1$ in fact is nonzero in $H^2(\BdeHam(\Sigma);\bR)$.
\end{rem}
\subsection{Non-vanishing results for classes  in $H^*(\BdSymp(\Sigma,\partial);\bZ)$} Let us first recall what we know about MMM-classes for symplectic flat surface bundles.  Morita observed  in \cite{morita1987characteristic} that the Bott vanishing theorem implies that $\kappa_i$ for $i>2$ vanishes in $H^{2i}(\BdDiff(\Sigma,\partial);\bQ)$. Hence it also vanishes in $H^{2i}(\BdSymp(\Sigma,\partial);\bQ)$. Kotschick and Morita in \cite{kotschick2005signatures} however proved that $\kappa_1\in H^{2}(\BdSymp(\Sigma,\partial);\bQ)$ is nonzero. One can also use \Cref{thm3} to prove their result about $\kappa_1$ (see \Cref{cor2} below). 

With integer coefficients, however, the author proved in \cite[Corollary 2.6]{nariman2015stable} that all the MMM-classes in the stable range are nonzero in $H^*(\BdDiff(\Sigma,\partial);\bZ)$. Exactly the same idea works to show that stable MMM-classes are also nonzero in $H^*(\BdSymp(\Sigma,\partial);\bZ)$. Here, we give a sketch of the argument and refer the reader to \cite[Theorem 2.4]{nariman2015stable} for further details.
\begin{proof}[Proof sketch of \Cref{thm5}] It is enough to prove that the quotient map 
\[
\iota: \dSymp(\Sigma,\partial)\to \text{MCG}(\Sigma,\partial),
\]
induces an injective map on cohomology with integer coefficients
\[
\iota^*: H^*(\mathrm{B} \text{MCG}(\Sigma,\partial);\bZ)\hookrightarrow H^*(\BdSymp(\Sigma,\partial);\bZ),
\]
in the stable range. Since the homology of the mapping class group is finitely generated in the stable range by the Madsen-Weiss theorem, Corollary 1 in \cite{milnor1983homology} implies that the injection of $\iota^*$ is equivalent to showing that for every prime $p$, the map
\[
H^*(\mathrm{B} \text{MCG}(\Sigma,\partial);\bF_p)\hookrightarrow H^*(\BdSymp(\Sigma,\partial);\bF_p),
\]
is injective in the stable range. Therefore, by the Madsen-Weiss theorem and \Cref{thm3}, it is enough to show the map
\[
H^*(\Omega^{\infty}_{\bullet}\text{MTSO}(2);\bF_p)\to H^*(\Omega^{\infty}_{\bullet}\text{MT}\theta;\bF_p),
\]
is injective, where $\text{MTSO}(2)$ is the Madsen-Tillmann spectrum (\cite{MR1856399}). To recall a definition of this spectrum, let $\gamma$ denote the tautological $2$-plane bundle over $\mathrm{BSL}_2(\bR)$. The Madsen-Tillmann spectrum can be described as the Thom spectrum of $-\gamma$.  

Note that the rotation matrices is a subgroup of the group of endomorphisms of the origin in the groupoid $\Gamma_2^{\text{vol}}$. Since endomorphisms of each object in $\Gamma_2^{\text{vol}}$ is a discrete group, we obtain the following maps
\[
\mathrm{B} {S^1}^{\delta}\xrightarrow{\eta} \BH\xrightarrow{\theta} \mathrm{BSL}_2(\bR).
\]
Using \cite[Lemma 3]{milnor1983homology}, one can see that $\theta\circ \eta$ induces homology isomorphisms with $\bF_p$-coefficients. The map $\theta\circ \eta$ induces a tangential structure and let $\text{MT}(\theta\circ \eta)$ denote the Thom spectrum of $(\theta\circ \eta)^*(-\gamma)$. Thus the map $\theta\circ \eta$ induces a  spectrum map from  $\text{MT}(\theta\circ \eta)$ to $\text{MTSO}(2)$. Given that  $\theta\circ \eta$ induces $\bF_p$-cohomology isomorphisms, the composition map 
\[
 H^*(\Omega^{\infty}_{\bullet}\text{MTSO}(2);\bF_p)
\to H^*(\Omega^{\infty}_{\bullet}\text{MT}\theta;\bF_p)\to  H^*(\Omega^{\infty}_{\bullet}\text{MT}(\theta\circ \eta);\bF_p),\]
is an isomorphism. Hence, the map
\[
H^*(\Omega^{\infty}_{\bullet}\text{MTSO}(2);\bF_p)\to H^*(\Omega^{\infty}_{\bullet}\text{MT}\theta;\bF_p),
\]
is injective.
\end{proof}
To study cohomology  classes in $H^*(\BdSymp(\Sigma,\partial);\bZ)$ other than MMM-classes, we need to find nontrivial classes in $H^*(\BH;\bZ)$ other than powers of the Euler class. Unfortunately, we still do not know if any of the exotic classes (e.g. Godbillon-Vey classes or Gelfand-Kalinin-Fuks classes \cite{MR0312531}) for foliations with transverse volume form are nontrivial. Hurder in \cite{MR712266} proved that for such foliations with the codimension larger than $2$ some of the exotic classes are nontrivial. Nonetheless, as we shall see below, the fact that the volume form $v\in H^2(\BH;\bR)$ is nontrivial provides us with a plethora of nontrivial invariants for the symplectic flat surface bundles. 

 Recall that the maps $e_n$ in \ref{eq9}  are defined integrally. We  write $e^p_2$ for the map induced on homology with $\bZ_{(p)}$-coefficients i.e. integers localized at the prime $p$. The statement of \Cref{thm2'} is implied by the first part of the following theorem.
\begin{thm}
(a) For $p>3$, the map $e^p_2$ is an isomorphism if $g(\Sigma)\geq 4$ and epimorphism if $g(\Sigma)\geq 3$. 
 (b) The map $e_3$ is an isomorphism with rational coefficients if $g(\Sigma)\geq 6$. 
\end{thm}
\begin{proof}
(a) Given the Thom isomorphism $H_*(\text{MT}\theta;\bZ)\cong H_{*+2}(\BH;\bZ)$ and \Cref{thm3}, it is enough to show that the map induced by the suspension map
\[
H_2(\Omega_{\bullet}^{\infty}\text{MT}\theta ;\bZ_{(p)})\to H_2(\text{MT}\theta; \bZ_{(p)}),
\]
is an isomorphism for $p>3$. Since $\dSymp(\Sigma,\partial)$ is perfect (\cite[Section 2.1]{kotschick2005signatures}) for $g(\Sigma)\geq 3$, the first homology of $\Omega_{\bullet}^{\infty}\text{MT}\theta$ is zero. Therefore by the Hurewicz theorem, we have the isomorphism
\[
\pi_2(\Omega_{\bullet}^{\infty}\text{MT}\theta)\xrightarrow{\cong} H_2(\Omega_{\bullet}^{\infty}\text{MT}\theta;\bZ).
\]
Recall that we have the following commutative diagram
\[
\begin{tikzcd}
\pi_2(\Omega^{\infty}_{\bullet}\text{MT}\theta)_{(p)}\arrow["\cong"]{r}\arrow["\cong"]{d}&\pi_2(\text{MT}\theta)_{(p)}\arrow["h"]{d}\\ H_2(\Omega^{\infty}_{\bullet}\text{MT}\theta;\bZ_{(p)})\arrow{r}&H_2(\text{MT}\theta;\bZ_{(p)}).
\end{tikzcd}
\] 
Hence, to show that the bottom horizontal map is an isomorphism, it suffices to prove that the right vertical map $h$ which is a Hurewicz map is an isomorphism. Note that $h$ is induced by the unit map from the localized sphere spectrum $\bS_{(p)}$ to the Eilenberg-Maclane spectrum $H\bZ_{(p)}$. We shall write this unit map as
\[
e:\bS_{(p)}\to H\bZ_{(p)}.
\]
Let $F_{(p)}$ denote the homotopy fiber of $e$. It is well-known (see e.g. \cite[Theorem 5.29]{hatcher2004spectral}) that the first nontrivial cohomology of $H\bZ_{(p)}$ in positive degrees appears in degree $2p-1$ and it is a $p$-torsion. Hence, the spectral sequence implies the first nontrivial cohomology group of $F_{(p)}$ in positive degrees appears in degree $2p-2$ and is a $p$-torsion. Therefore, by universal coefficient theorem the first nontrivial homology group of $F_{(p)}$ in positive degrees appears in degree $2p-3$. Thus, the map $e$ is $(2p-4)$-connected. Hence, for $2p-4>2$, the map $h$
\[
h: \pi_2(\text{MT}\theta)_{(p)}\to H_2(\text{MT}\theta; \bZ_{(p)}),
\]
induces an isomorphism.

(b) To prove that the map 
\[
e_3: H_3(\BdSymp(\Sigma,\partial);\bQ)\to H_5(\BH;\bQ),
\]
is an isomorphism for $g(\Sigma)\geq 6$, recall it suffices to show that the suspension map
\[
H_3(\Omega_{\bullet}^{\infty}\text{MT}\theta ;\bQ)\to H_3(\text{MT}\theta; \bQ),
\]
is an isomorphism.  To do so, consider the commutative diagram 
\[
\begin{tikzcd}
\pi_3(\Omega_{\bullet}^{\infty}\text{MT}\theta)\otimes \bQ\arrow["\cong"]{r}\arrow[two heads]{d}&\pi_3(\text{MT}\theta)\otimes\bQ\arrow["\cong"]{d}\\ H_3(\Omega_{\bullet}^{\infty}\text{MT}\theta;\bQ)\arrow{r}&H_3(\text{MT}\theta;\bQ).
\end{tikzcd}
\] 
The left vertical map is surjective by the rational Hurewicz theorem because $$H_1(\Omega_{\bullet}^{\infty}\text{MT}\theta;\bZ)=0.$$ The top horizontal map is an isomorphism by definition and the right vertical map is an isomorphism again by the rational Hurewicz theorem. Hence, the bottom horizontal map has to be an isomorphism.
\end{proof}
\begin{rem}
It seems possible to use the Adams spectral sequence to analyze what happens in part (a) at the primes 2 and 3, but we have not pursued this point.
\end{rem}
In particular the part (a) of the theorem implies that $H_2(\BdSymp(\Sigma,\partial);\bQ)\cong H_4(\BH;\bQ)$ for $g(\Sigma)\geq 4$. To find new nontrivial invariants of flat symplectic surface bundles, we shall prove $H_4(\BH;\bQ)$ is highly nontrivial as a $\bQ$-vector space. 

We define three classes in the cohomology of  $\BH$ with different coefficients. The first is induced by the class $e^2$ as a cohomology class in $H^4(\BH;\bQ)$ which gives rise to the first MMM-class. The second class is $ev$ as a real cohomogy class in $H^4(\BH;\bR)$. And the third  is a secondary class induced by the vanishing of $v^2=0$ in $H^4(\BH;\bR)$ as follows. Consider the map $v: \BH\to K(\bR,2)$. The class $v$ induces a map
\[
\tilde{v^2}:H_4(\BH;\bQ)\to H_4(K(\bR,2);\bQ)\cong S^2_{\bQ}\bR
\]
The class $v^2\in H^4(\BH;\bR)$ can be described as follows
\[
m\circ \tilde{v^2}: H_4(\BH;\bQ)\to S^2_{\bQ}\bR\xrightarrow{m} \bR,
\]
where $m$ is the natural map given by multiplication. Therefore, $\tilde{v^2}$ maps $H_4(\BH;\bQ)$ onto $\Ker(m)$.

\begin{thm}\label{thm9}
The map
\[
\begin{tikzcd}
(e^2,ev,\tilde{v^2}): H_4(\BH;\bQ)\arrow[two heads, ""]{r}& \bQ\oplus \bR\oplus \Ker(m)\cong \bQ\oplus S^2_{\bQ}\bR,
\end{tikzcd}
\] 
is surjective.
\end{thm}
\begin{proof}
Consider the map
\[
\theta\times v: \BH\to \mathrm{BSL}_2(\bR)\times K(\bR,2),
\]
and let $  \overline{\overline{\BH}}$ denote the homotopy fiber of the map $\theta\times v$. We want to determine the image of the map induced by $\theta\times v$ on the fourth homology groups. McDuff proved (see \cite[Theorem 6.1]{mcduff1987applications}) that $ \overline{\overline{\BH}}$ is $2$-connected and as she observed in \cite[Corollary 1.3]{MR707329},  the result of Banyaga in \cite{MR490874} implies that \begin{equation}\label{ss}\pi_3(\overline{\overline{\BH}})\cong \bR.\end{equation} The geometric meaning of the  space $ \overline{\overline{\BH}}$, by the general theory of Haefliger structures in \cite{haefliger1971homotopy},  is that it classifies foliated trivialized $2$-plane bundle with a transverse volume form whose volume form is exact.

Let $x\in H^2(K(\bR,2);\bR)=\text{Hom}(\bR,\bR)$ be the fundamental class given by the identity. Using the calculation in \ref{eq7}, one can see that $x^2\in H^4(K(\bR,2);\bR)=\text{Hom}(S^2_{\bQ}\bR,\bR)$ corresponds to the natural map
\[
m: S^2_{\bQ}\bR\to \bR.
\]
 On the other hand the pullback of $x$ to $\BH$ is the volume form $v$, hence $v^2=0$. Since the map $\theta\times v$ is $3$-connected (\cite[Theorem 6.1]{mcduff1987applications}), the space $\BH$ is simply connected. Therefore, in the cohomology Serre spectral sequence for the fibration
\begin{equation}\label{eq8}\begin{gathered}
 \overline{\overline{\BH}}\to \BH\to \mathrm{BSL}_2(\bR)\times K(\bR,2),
\end{gathered}\end{equation}
the class $x^2$ is not hit by $d^2$ and $d^3$. Hence, there should be a class in $a\in H^3( \overline{\overline{\BH}};\bR)=\text{Hom}(\bR,\bR)$ that transgresses to $x^2$ i.e. $d^4(a)=x^2$. In fact $a$ in \cite[Lemma 4]{MR707329} is constructed by differential forms on foliated trivialized $2$-plane bundle with a transverse volume form whose volume form is exact. Therefore $a$ is an $\bR$-linear map in $\text{Hom}(\bR,\bR)$ and by scaling the isomorphism \ref{ss}, we can assume that $a$ corresponds to the identity in $\text{Hom}(\bR,\bR)$.

\begin{figure}[ht]\label{sss}
\[
\begin{tikzpicture}
\draw [<-] (-0.6,6) -- (-0.6,-0.5);
\draw [->] (-1,0)-- (9,0);
\node  at (0.2,0.2) {\small$\bQ$};
\node  at (0.2,-0.2) {$0$};
\node at (1.2,0.2) {\small$0$};
\node at (1.2,-0.2) {\small$1$};

\node at (2.5,0.2) {\small$\bQ\oplus\bR$};
\node at (2.5,-0.2) {\small$2$};

\node at (3.7,0.2) {\small$0$};
\node at (3.7,-0.2) {\small$3$};

\node at (5.5,0.2) {\small$\bQ\oplus\bR\oplus S^2_{\bQ}\bR$};
\node at (5.5,-0.2) {\small$4$};

\node at (7.2,0.2) {\small$0$};
\node at (7.2,-0.2) {\small$5$};

\node at (9.2,-0.2) {\small$p$};

\node at (0.2,1.2) {\small$0$};
\node at (-1,1.2) {\small$1$};
\node at (-1,0.2) {\small$0$};

\node at (0.2,2.2) {\small$0$};
\node at (-1,2.2) {\small$2$};

\node at (0.2,3.2) {\small$\bR$};
\node at (-1,3.2) {\small$3$};

\node at (0.2,4.2) {\small$H_4(\overline{\overline{\BH}})$};
\node at (-1,4.2) {\small$4$};

\node at (0.2,5.2) {\small$$};
\node at (-1,5.2) {\small$5$};
\node at (-1,6.2) {\small$q$};

\node at (1.2,2.2) {\small$0$};
\node at (2.5,2.2) {\small$0$};
\node at (3.7,2.2) {\small$0$};
\node at (5.5,2.2) {\small$0$};
\node at (7.2,2.2) {\small$0$};
\node at (1.2,3.2) {\small$0$};
\node at (2.5,3.2) {\small$\bR\oplus (\bR\otimes\bR)$};
\node at (1.2,4.2) {\small$0$};

\node at (1.2,1.2) {\small$0$};
\node at (2.5,1.2) {\small$0$};
\node at (3.7,1.2) {\small$0$};
\node at (5.5,1.2) {\small$0$};
\node at (7.2,1.2) {\small$0$};
\draw [->,red] (5.5,0.4)--(0.35,3.2);
\draw [->,red] (2.5,3.4)--(1,4);
\node at (2,2.1) {\small$\textcolor{red}{d_4}$};
\node at (2,3.9) {\small$\textcolor{red}{d_2}$};
\end{tikzpicture}\]
\caption{ Second page of the homology spectral sequence}
\end{figure}
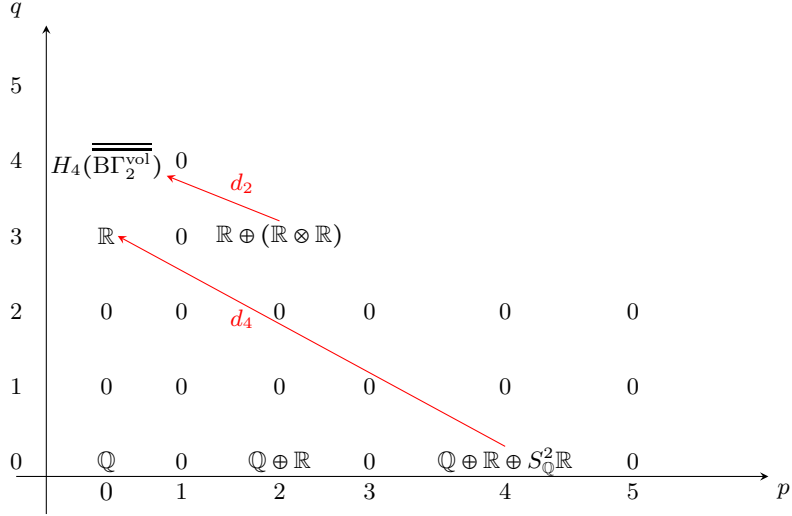
Now the homology Serre spectral sequence for the fibration \ref{eq8} looks like Figure $1$.
Since the transgression $d^4$ in the cohomology spectral sequence is induced by the differential $d_4$ in the homology spectral sequence, we deduce that 
\[
d_4: \bQ\oplus\bR\oplus S^2_{\bQ}\bR\xrightarrow{proj} S^2_{\bQ}\bR\xrightarrow{m} \bR,
\]
where the first map is the projection to the third factor. Hence, the kernel of $d_4$ is
\[
\Ker(d_4)=\bQ\oplus \bR\oplus \text{Ker}(m)\cong \bQ\oplus S^2_{\bQ}\bR.
\]
Therefore, we have 
\begin{align*}
H_2(\BH;\bQ)&\cong\bQ\oplus \bR,\\ H_3(\BH;\bQ)&\cong 0,\\ 0\to\Coker(d_2)\to H_4(\BH;\bQ)&\to\bQ\oplus S^2_{\bQ}\bR\to 0.
\end{align*}
\end{proof}
\begin{cor}\label{cor2}
For $g(\Sigma)\geq 3$, there is a surjective map
\[
\begin{tikzcd}
H_2(\BdSymp(\Sigma,\partial);\bQ)\arrow[two heads]{r}& \bQ\oplus S^2_{\bQ}\bR,
\end{tikzcd}
\] 
and for $g(\Sigma)\geq 4$, we have a short exact sequence
\[
0\to\Coker(d_2)\to H_2(\BdSymp(\Sigma,\partial);\bQ)\to\bQ\oplus S^2_{\bQ}\bR\to 0.
\]
\end{cor}
Hence as a corollary, similar to \Cref{prop1} we obtain  the main theorem of Kotschick and Morita   in \cite{kotschick2004characteristic}:
\begin{cor}\label{cor4}
There is a surjective map
\[
\begin{tikzcd}
H_{2k}(\BdSymp(\Sigma,\partial);\bQ) \arrow[two heads, ""]{r} & \bQ\oplus S^2_{\bQ}\bR\oplus \cdots\oplus S^k(S^2_{\bQ}\bR).
\end{tikzcd}
\]
for $g(\Sigma)\geq 3k$.
\end{cor}
\begin{rem}
Note that the above invariants can be defined on $H_{2k}(\BdSymp(\Sigma);\bQ)$ when the surface $\Sigma$ is a closed surface. Therefore \Cref{cor4} also holds for closed surfaces.
\end{rem}
\subsection{Relation to the Kotschick-Morita classes} Note that all the nontrivial invariants constructed in  \Cref{cor4}, are induced from the map
\[
H_2(\BdSymp(\Sigma);\bQ)\to \bQ\oplus S^2_{\bQ}\bR.
\]
Recall that the first $\bQ$ summand is induced by $\kappa_1$. There are two ways to describe the map to the second factor. One way is what Kotschick and Morita did in \cite[Section 1]{kotschick2004characteristic}  which is roughly as follows. The extended flux homomorphism gives rise to a twisted cohomology class
\[
[ \widetilde{\text{Flux}}]\in H^1(\mathrm{BSymp}^{\delta}(\Sigma_g); H^1(\Sigma_g;\bR)).
\]
Then the square of this class lives in 
\[
[ \widetilde{\text{Flux}}]^2\in H^2(\mathrm{BSymp}^{\delta}(\Sigma_g); H^1(\Sigma_g;\bR)\otimes H^1(\Sigma_g;\bR)).
\]
Now one can use the intersection from $\iota:H^1_{\bR}\otimes H^1_{\bR}\to \bR$ to obtain a class \begin{equation}\label{alpha}\alpha\in H^2(\mathrm{Symp}^{\delta}(\Sigma_g);\bR).\end{equation}  As explained in \cite[Definition 1]{kotschick2004characteristic}, one can refine the intersection form $\iota$ by the {\it discontinuous} cup product $\tilde{\iota}$ so that
\[
\iota: H^1_{\bR}\otimes H^1_{\bR}\xrightarrow{\tilde{\iota}} S^2_{\bQ}\bR\xrightarrow{m}\bR.
\]
Therefore, the class $\alpha$ is induced from a class $\tilde{\alpha}\in H^2(\mathrm{BSymp}^{\delta}(\Sigma_g); S^2_{\bQ}\bR)$. The class $\tilde{\alpha}$ induces a map
\[
\tilde{\alpha}: H_2(\BdSymp(\Sigma);\bQ)\to  S^2_{\bQ}\bR.
\]

The  way  we would like to think about these $S^2_{\bQ}\bR$-valued characteristic classes is to describe their evaluation on  a class in $a\in H_2(\BdSymp(\Sigma);\bQ)$. Recall that we can represent the class $a$ as the image of a map  $\Sigma'\to \BdSymp(\Sigma)$ for some surface $\Sigma'$. Therefore the class $a$ gives rise to a symplectic flat surface bundle $\Sigma\to E\to \Sigma'$.  By the general theory of Haefliger spaces, the foliation on the total space $E$ gives rise to a map  $f:E\to\BH$ that is well-defined up to homotopy. Consider the diagram
\[
\begin{tikzpicture}[node distance=1.4cm, auto]
  \node (C) {$ \Sigma$};
  \node (E) [right of=C, node distance=1.4cm]{$ E$};
  \node (F) [below of=E]{$ \Sigma'.$};
    \node (A) [right of=E]{$ \BH$};
        \node (B) [right of=A,node distance=2.4cm]{$ K(\bR,2)$};
    \node (D) [right of=B,node distance=2.4cm]{$ K(\bR,4)$};
  \draw [->] (C) to node {$$} (E);
  \draw [<-] (F) to node {$\pi$} (E);
    \draw [->] (E) to node {$f$} (A);
  \draw [->] (A) to node {$e+v$} (B);
    \draw [->] (B) to node {$\text{Sq}^2$} (D);
 \end{tikzpicture}
\]
To obtain a number associated to the fundamental class of $E$, we take the induced map on homology by $ (e+v)^2\circ f$:
\[
H_4(E;\bQ)\xrightarrow{f_*} H_4(\BH;\bQ)\xrightarrow{(e+v)_*} S^2_{\bQ}\bR \xrightarrow{m} \bR.
\]
Hence, we can associate to $a$ the number $((e+v)^2\circ f)_*([E])\in \bR$. We can also refine this class by assigning to $a$  the element $((e+v)\circ f)_*([E])\in S^2_{\bQ}\bR$. 

One can see that these two ways of constructing invariants of symplectic flat surface bundles agree up to sign using an observation due to Kawazumi (\cite[Section 7]{kotschick2004characteristic}). He noted that the contraction formula  (\cite[Theorem 6.2]{kawazumi2001primary}) implies $$\pi_!((e+v)^2)= -\alpha.$$ Therefore we have  $((e+v)^2\circ f)_*([E])=-\alpha(a)$. In fact one can use the contraction formula  with more care to  show that $$((e+v)\circ f)_*([E])=-\tilde{\alpha}\in S^2_{\bQ}\bR.$$

So in order to relate the class $\tilde{\alpha}$ to our calculation in \Cref{thm9}, we need to relate the map
\[
(e+v)_*: H_4(\BH;\bQ)\to S^2_{\bQ}\bR,
\]
to the map we obtained in the spectral sequence in Figure $1$
\[
H_4(\BH;\bQ)\to E^{4,0}_{\infty}=\Ker(d_4)\cong \bQ\oplus \bR\oplus \Ker(m: S^2_{\bQ}\bR\to \bR).
\]
Recall that this map in the spectral sequence is induced by the map
\[
\BH\xrightarrow{(e,v)} K(\bQ,2)\times K(\bR,2).
\]
Therefore, by factoring the map $(e+v)$ as follows
\[
\BH\xrightarrow{(e,v)}  K(\bQ,2)\times K(\bR,2)\xrightarrow{\text{sum}}K(\bR,2),
\]
we deduce  that $(e+v)_*$ is given by the composition
\[
 H_4(\BH;\bQ)\to E^{4,0}_{\infty}\to \bR\oplus \Ker(m: S^2_{\bQ}\bR\to \bR)\cong S^2_{\bQ}\bR.
\]
Now given the relation between these two points of view, we prove \Cref{thm10}.
\begin{proof}[Proof of \Cref{thm10}]
Recall from the proof of \Cref{thm3'}, we can consider the following composition of maps
{\small $$
 \mathrm{B}\deHam(\Sigma)\simeq \BdHam(\Sigma)\hcoker \Symp(\Sigma)\longrightarrow \mathrm{B}\bR^{\delta}\hker\text{Bun}(T\Sigma,{\beta}^*\gamma)\hcoker \Symp(\Sigma) \to \mathrm{BB}\bR^{\delta} \simeq K(\bR,2),
$$}
which gives rise to a cohomology class $a\in H^2(\BdeHam(\Sigma);\bR)$. Therefore, we have a homotopy commutative diagram
\begin{equation}\label{R}
  \begin{gathered}
 \begin{tikzpicture}[node distance=2cm, auto]
  \node (A) {$\BdeHam(\Sigma, \text{rel }D^2)$};
  \node (B) [below of=A]{$\Omega^{\infty}\text{MT}\beta$};
  \node (C) [right of=A, node distance=3cm]{$\BdeHam(\Sigma)$};
  \node (D) [right of=B, node distance=3cm]{$\mathrm{B}\bR^{\delta}\hker \Omega^{\infty}\text{MT}\beta$};
  \node (E) [right of=C, node distance=3cm]{$K(\bR,2)$};
  \node (F) [below of=E]{$K(\bR,2),$};
  \draw [->] (A) to node {$$} (C);
  \draw [->] (B) to node {$$} (D);
  \draw [->] (A) to node {$$} (B);
  \draw [->] (C) to node {$$} (D);
  \draw [->] (C) to node {$a$} (E);
  \draw [->] (E) to node {$=$} (F);
  \draw [->] (D) to node {$$} (F);
 \end{tikzpicture}
   \end{gathered}
 \end{equation}
 where the two first vertical maps are homology isomorphisms in the stable range. Since the bottom row is a fibration sequence, the class $a\in H^2(\BdeHam(\Sigma, \text{rel }D^2);\bR)$ vanishes for $g(\Sigma)\geq 3$. Therefore, there exists a map 
 \[
 \BdeHam(\Sigma, \text{rel }D^2)\to \text{hofib}(a),
 \]
 that induces a homology isomorphism in the stable range. Hence, we need to show that $a=\frac{\kappa_1}{4-4g(\Sigma)}$. 
 
  Consider the universal $\Sigma$-bundle 
\begin{equation}\label{12}
\begin{gathered}
\begin{tikzpicture}[node distance=2cm, auto]
  \node (A) {$ \Sigma$};
  \node (B) [ right of=A] {$\Sigma\hcoker \deHam(\Sigma)$};
  \node (D) [below of=B, node distance=1.3cm] {$\BdeHam(\Sigma),$};
  \draw [->] (A) to node {$$} (B);
  \draw [->] (B) to node {$\pi$} (D);
\end{tikzpicture}
\end{gathered}
\end{equation}
whose holonomy lies in $\deHam(\Sigma)$. With abuse of notation, let the class $e+v\in H^2(\Sigma\hcoker \deHam(\Sigma);\bR)$ also denote the sum of the Euler class of the vertical tangent bundle and the transverse volume form.  Note that the Serre spectral sequence calculating the cohomology of $H^*(\Sigma\hcoker \deHam(\Sigma);\bR)$ collapses (see \cite[Proposition 3.1]{morita1987characteristic}). Therefore we have
\[
H^2(\Sigma\hcoker \deHam(\Sigma);\bR)\cong E^{2,0}_2\oplus E^{1,1}_2\oplus E^{0,2}_2.
\]
The projection of $e+v$ to $E^{0,2}_2$ is zero since the volume is normalized and the restriction of $e+v$ to each fiber is an exact form. The projection of $e+v$ to $E^{1,1}_2$ is the extended Flux by \cite[Lemma 8]{kotschick2005signatures} and therefore by definition of the extended Hamiltonians, the projection of $e+v$ to $E^{1,1}_2$ is zero. As we shall show in the claim below, we have $\pi^*(a)=e+v$. Hence, $a$ is the unique cohomology class in $ H^2(\BdeHam(\Sigma);\bR)$ that $\pi^*(a)=e+v$. 

\noindent{\bf Claim:} The class $\pi^*(a)$ is equal to $e+v\in H^2(\Sigma\hcoker \deHam(\Sigma);\bR)$.

\noindent{\it Proof of the claim:} Since on the total space $\Sigma\hcoker \deHam(\Sigma)$ there exists a codimension $2$ Haefliger structure with a transverse volume form, we obtain a bundle map
\[
\begin{tikzpicture}[node distance=2.2cm, auto]
  \node (A) {$ T\pi$};
  \node (B) [ below of=A, node distance=1.3cm] {$\Sigma\hcoker \deHam(\Sigma)$};
  \node (D) [right of=B] {$\BH,$};
    \node (C) [right of=A] {$\theta^*(\gamma)$};
  \draw [->] (A) to node {$$} (B);
  \draw [->] (A) to node {$$} (C);
    \draw [->] (B) to node {$$} (D);
    \draw [->] (C) to node {$$} (D);
\end{tikzpicture}
\]
where $T\pi$ is the vertical tangent bundle for the surface bundle (\ref{12}). Hence to prove the claim, it is enough to show that the following diagram is homotopy commutative
\[
\begin{tikzpicture}[node distance=2.2cm, auto]
  \node (A) {$\Sigma\hcoker \deHam(\Sigma)$};
  \node (B) [ below of=A, node distance=1.5cm] {$\BdeHam(\Sigma)$};
  \node (D) [right of=B] {$K(\bR,2).$};
    \node (C) [right of=A] {$\BH$};
  \draw [->] (A) to node {$\pi$} (B);
  \draw [->] (A) to node {$$} (C);
    \draw [->] (B) to node {$a$} (D);
    \draw [->] (C) to node {$e+v$} (D);
\end{tikzpicture}
\]
To do so, we first give a different description of the map $\Sigma\hcoker \deHam(\Sigma)\to K(\bR,2)$ induced by $e+v$. Recall from (\ref{11}), that our point-set model for $\BdeHam(\Sigma)$ is $\BdHam(\Sigma)\hcoker\Symp(\Sigma)$. Also recall from the diagram \ref{W} that there is a $\Symp(\Sigma)$-equivariant map from $\BdHam(\Sigma)$ to $X$ which is a homology isomorphism. Hence, we have a map between the $\Sigma$-bundles
\[
\begin{tikzpicture}[node distance=3.2cm, auto]
  \node (A) {$\Sigma\hcoker \deHam(\Sigma)$};
  \node (B) [ below of=A, node distance=1.5cm] {$\BdeHam(\Sigma)$};
  \node (D) [right of=B] {$X\hcoker \Symp(\Sigma),$};
    \node (C) [right of=A] {$(X\times \Sigma)\hcoker \Symp(\Sigma)$};
  \draw [->] (A) to node {$\pi$} (B);
  \draw [->] (A) to node {$$} (C);
    \draw [->] (B) to node {$$} (D);
    \draw [->] (C) to node {$\pi'$} (D);
\end{tikzpicture}
\]
where the horizontal maps induce homology isomorphisms. Thus, it is enough to prove the claim for the $\Sigma$-bundle $\pi'$. From (\ref{X}), we have a homotopy commutative diagram with $\Symp(\Sigma)-$equivariant maps 
\[
 \begin{tikzpicture}[node distance=1.5cm, auto]
  \node (B) {$X$};
  \node (C) [right of=B, node distance=4.2cm]{$K(\bR,2)$};
  \node (E) [below of=B]{$ \text{Bun}_{0}(T\Sigma,{\theta}^*\gamma)$};
  \node (F) [right of=E, node distance=4.2cm]{$ \text{Map}_{0}(\Sigma,K(\bR,2)).$};
  \draw [->] (B) to node {$$} (C);
    \draw [->] (C) to node {$$} (F);
    \draw [->] (B) to node {$g$} (E);
  \draw [->] (E) to node {$-\circ (e+v)$} (F);
 \end{tikzpicture}
  \]
Therefore, the $\Symp(\Sigma)-$equivariant map induced by $e+v$ from $X$ to $ \text{Map}_{0}(\Sigma,K(\bR,2))$ factors through constant maps that are identified with $K(\bR,2)$. Hence, we have a commutative diagram with  $\Symp(\Sigma)-$equivariant maps
\[
 \begin{tikzpicture}[node distance=1.5cm, auto]
  \node (B) {$X\times \Sigma$};
  \node (C) [right of=B, node distance=5.2cm]{$ \text{Map}_{0}(\Sigma,K(\bR,2))\times \Sigma$};
  \node (E) [below of=B]{$X$};
  \node (F) [right of=E, node distance=5.2cm]{$K(\bR,2),$};
  \draw [->] (B) to node {$g\circ (e+v)\times id$} (C);
    \draw [->] (C) to node {$$} (F);
    \draw [->] (B) to node {$$} (E);
  \draw [->] (E) to node {$$} (F);
 \end{tikzpicture}
  \]
where the left vertical map is projection and the right vertical map is the evaluation map. Since the action of $\Symp(\Sigma)$ on $K(\bR,2)$ is trivial, we obtain the homotopy commutative diagram
\[
 \begin{tikzpicture}[node distance=1.5cm, auto]
  \node (B) {$(X\times \Sigma)\hcoker \Symp(\Sigma)$};
  \node (C) [right of=B, node distance=6.8cm]{$(\text{Map}_{0}(\Sigma,K(\bR,2))\times \Sigma)\hcoker \Symp(\Sigma)$};
  \node (E) [below of=B]{$X\hcoker \Symp(\Sigma)$};
  \node (F) [right of=E, node distance=6.8cm]{$K(\bR,2).$};
  \draw [->] (B) to node {$g\circ (e+v)\times id$} (C);
    \draw [->] (C) to node {$$} (F);
    \draw [->] (B) to node {$\pi'$} (E);
  \draw [->] (E) to node {$a$} (F);
 \end{tikzpicture}
  \]
So the $\pi'^*(a)$ is the same as the class induced by $e+v$ on the total space of the surface bundle $\pi'$. $ \blacksquare$

Therefore, we have
\begin{align*}
\pi_!(e(e+v))&= \pi_!(e\pi^*(a))\\
\kappa_1+\pi_!(ev)&= (2-2g)a.
\end{align*}
From Kawazumi's argument ([KM07, Section 7]) we have $\pi_!(ev)=-(\kappa_1+\alpha)/2$ where $\alpha$ is the class defined in (\ref{alpha}). By definition of the class $\alpha$, it is zero in $H^2(\BdeHam(\Sigma);\bR)$ because it is defined by the square of the extended flux which vanishes on the extended Hamiltonian group. Therefore, we obtain 
\[
a=\frac{\kappa_1+\pi_!(ev)}{2-2g(\Sigma)}=\frac{\kappa_1}{4-4g(\Sigma)}.
\]
 \end{proof}

\subsection{Discussion about higher dimensions} Since the method of Kotschick and Morita heavily relies on the theory of surfaces, it is not obvious how to generalize their calculations to higher dimensions. One possible generalization though of our method is to consider the volume preserving diffeomorphisms of high dimensional analogue of surfaces. So let $(\W,\omega)$ denote a pair of a manifold diffeomorphic to $\SnSn \backslash \text{int} D^{2n}$ and a volume form $\omega$. Let $\dDiff_{\omega}(\W,\partial)$ denote the discrete $\omega$-preserving compactly supported diffeomorphisms of $\W\backslash \partial \W$.

To introduce the relevant tangential structure in this case, let $\mathrm{BSL}_{2n}(\bR)\langle n\rangle$ and $\mathrm{B}\Gamma^{\text{vol}}_{2n}\langle n\rangle$ be the $n$-connected covers of $\mathrm{BSL}_{2n}(\bR)$ and $\mathrm{B}\Gamma^{\text{vol}}_{2n}$ respectively. Given  that $\theta$ is a $2n$-connected map \cite{haefliger1971homotopy}, we have the following homotopy pullback diagram
   \[
 \begin{tikzpicture}[node distance=2cm, auto]
  \node (A) {$\mathrm{B}\Gamma^{\text{vol}}_{2n}$};
  \node (B) [below of=A] {$\mathrm{BSL}_{2n}(\bR)$.};
   \node (C) [left of= A, node distance=3.2cm] {$\mathrm{B}\Gamma^{\text{vol}}_{2n}\langle n\rangle$};
  \node (D) [below of=C] {$\mathrm{BSL}_{2n}(\bR)\langle n\rangle$};
  \draw [->] (C) to node {$\theta\langle n\rangle$}(D);
  \draw [->] (A) to node {$\theta$}(B);
  \draw [->] (C) to node {$\nu$}(A);
  \draw [->] (D) to node {$\nu\langle n\rangle$}(B);
\end{tikzpicture}
\]
\begin{defn}
Let $\gamma$ be the tautological bundle over $\mathrm{BSL}_{2n}(\bR)$. Let $\text{MT}\theta\langle n\rangle$ denote the Thom spectrum of $(\theta\circ\nu)^*(-\gamma)$. 
\end{defn}
Using the same idea as \Cref{sec2} and \cite[Theorem 1.4]{galatius2014homological}, one can show that $H_*(\BdDiff_{\omega}(\W,\partial);\bZ)$ is independent of $g$ as long as $*\leq (g-3)/2$. Moreover, there is a map 
\[
\BdDiff_{\omega}(\W,\partial)\to \Omega^{\infty}\text{MT}\theta\langle n\rangle,
\]
that induces a homology isomorphism in the stable range onto the connected component that it hits.  Similar to \Cref{prop1}, we obtain a surjective map 
\[
\begin{tikzcd}
H_*(\BdDiff_{\omega}(\W,\partial);\bQ)\arrow[two heads]{r} & H_{*+2n}(\mathrm{B}\Gamma^{\text{vol}}_{2n}\langle n\rangle;\bQ),
\end{tikzcd}
\]
for $*\leq (g-3)/2$. 

In order to detect nontrivial classes in $H_*(\mathrm{B}\Gamma^{\text{vol}}_{2n}\langle n\rangle;\bQ)$ we can use a fiber sequence similar  to \ref{eq8}. Let $\overline{\overline{\mathrm{B}\Gamma^{\text{vol}}_{2n}}}$ denote the homotopy fiber of 
\[
\mathrm{B}\Gamma^{\text{vol}}_{2n}\xrightarrow{(\theta, v)} \mathrm{BSL}_{2n}(\bR)\times K(\bR,2n).
\]
McDuff showed that $\overline{\overline{\mathrm{B}\Gamma^{\text{vol}}_{2n}}}$ is $2n$-connected. Hence, we have a fiber sequence
\begin{equation}\label{eq10}\begin{gathered}
\overline{\overline{\mathrm{B}\Gamma^{\text{vol}}_{2n}}}\to \mathrm{B}\Gamma^{\text{vol}}_{2n}\langle n\rangle\to \mathrm{BSL}_{2n}(\bR)\langle n\rangle\times K(\bR,2n).
\end{gathered}\end{equation}
Similar to \cite[Lemma 4]{MR707329}, there is a differential form $a\in H^{4n-1}(\overline{\overline{\mathrm{B}\Gamma^{\text{vol}}_{2n}}};\bR)$ that transgresses to $x^2$ where $x$ is the fundamental class in $H^{2n}(K(\bR,2n);\bR)$. Therefore in the homology spectral sequence for the fibration \ref{eq10}, we have a map
\[
S^2_{\bQ}\bR\hookrightarrow E^{4n,0}_{4n}\xrightarrow{d_{4n}} H_{4n-1}(\overline{\overline{\mathrm{B}\Gamma^{\text{vol}}_{2n}}};\bQ),
\]
so that the map 
\[
S^2_{\bQ}\bR\to H_{4n-1}(\overline{\overline{\mathrm{B}\Gamma^{\text{vol}}_{2n}}};\bQ)\xrightarrow{a}\bR,
\]
is the multiplication map $m:S^2_{\bQ}\bR\to \bR$, but it is not clear to the author whether $\Ker(m)\subset \Ker(d_{4n})$.
\begin{problem}
Prove or disprove that the map 
\[
H_{2n}(\BdDiff_{\omega}(\WW,\partial);\bQ)\to \Ker(m: S^2_{\bQ}\bR\to \bR),
\]
is surjective for $n\leq (g-3)/4$.
\end{problem}
\bibliographystyle{alpha}
\bibliography{reference}

\begin{thebibliography}{GMTW09}

\bibitem[Ban78]{MR490874}
Augustin Banyaga.
\newblock Sur la structure du groupe des diff\'eomorphismes qui pr\'eservent
  une forme symplectique.
\newblock {\em Comment. Math. Helv.}, 53(2):174--227, 1978. MR490874, Zbl 0393.58007

\bibitem[Ban97]{MR1445290}
Augustin Banyaga.
\newblock {\em The structure of classical diffeomorphism groups}, volume 400 of
  {\em Mathematics and its Applications}.
\newblock Kluwer Academic Publishers Group, Dordrecht, 1997. MR1445290, Zbl 0874.58005

\bibitem[Bow10]{bowden2010two}
Jonathan Bowden.
\newblock {\em Two-dimensional foliations on four-manifolds}.
\newblock PhD thesis, LMU, 2010. Zbl 1296.53003

\bibitem[Bow11]{MR2826937}
Jonathan Bowden.
\newblock Flat structures on surface bundles.
\newblock {\em Algebr. Geom. Topol.}, 11(4):2207--2235, 2011. MR2826937, Zbl 0185.32901

\bibitem[EE69]{earle1969fibre}
Clifford~J Earle and James Eells.
\newblock A fibre bundle description of teichm{\"u}ller theory.
\newblock {\em Journal of Differential Geometry}, 3(1-2):19--43, 1969.  MR0276999

\bibitem[ES70]{MR0277000}
C.~J. Earle and A.~Schatz.
\newblock Teichm\"uller theory for surfaces with boundary.
\newblock {\em J. Differential Geometry}, 4:169--185, 1970. MR0277000, Zbl 0194.52802

\bibitem[GKF72]{MR0312531}
I.~M. Gelfand, D.~I. Kalinin, and D.~B. Fuks.
\newblock The cohomology of the {L}ie algebra of {H}amiltonian formal vector
  fields.
\newblock {\em Funkcional. Anal. i Prilo\v zen.}, 6(3):25--29, 1972. MR0312531, Zbl 0259.57023

\bibitem[GMTW09]{galatius2009homotopy}
S{\o}ren Galatius, Ib~Madsen, Ulrike Tillmann, and Michael Weiss.
\newblock The homotopy type of the cobordism category.
\newblock {\em Acta mathematica}, 202(2):195--239, 2009.  MR2506750, Zbl 1221.57039

\bibitem[GRW14]{galatius2012stable}
S{\o}ren Galatius and Oscar Randal-Williams.
\newblock Stable moduli spaces of high dimensional manifolds.
\newblock {\em Acta Mathematica}, 2014. MR3207759, Zbl 1377.55012

\bibitem[GRW17]{galatius2014homological}
S{\o}ren Galatius and Oscar Randal-Williams.
\newblock Homological stability for moduli spaces of high dimensional
  manifolds. {I}.
\newblock {\em Journal of the American Mathematical Society}, 2017. MR3718454, Zbl 06802407

\bibitem[Hae71]{haefliger1971homotopy}
Andr{\'e} Haefliger.
\newblock Homotopy and integrability.
\newblock In {\em Manifolds-Amsterdam 1970}, pages 133--163. Springer, 1971.  MR0285027, Zbl 0215.52403

\bibitem[Har83]{harer1983second}
John Harer.
\newblock The second homology group of the mapping class group of an orientable
  surface.
\newblock {\em Inventiones Mathematicae}, 72(2):221--239, 1983.  MR0700769, Zbl 0533.57003

\bibitem[Har85]{harer1985stability}
John~L Harer.
\newblock Stability of the homology of the mapping class groups of orientable
  surfaces.
\newblock {\em Annals of Mathematics}, pages 215--249, 1985. MR0786348, Zbl 0579.57005

\bibitem[Hat04]{hatcher2004spectral}
Allen Hatcher.
\newblock Spectral sequences in algebraic topology.
\newblock {\em Unpublished book project, available on http://www. math.
  cornell. edu/hatcher/SSAT/SSATpage. html}, 2004.

\bibitem[Hur83]{MR712266}
Steven Hurder.
\newblock Global invariants for measured foliations.
\newblock {\em Trans. Amer. Math. Soc.}, 280(1):367--391, 1983. MR712266, Zbl 0551.57015

\bibitem[KM]{kawazumi2001primary}
Nariya Kawazumi and Shigeyuki Morita.
\newblock The primary approximation to the cohomology of the moduli space of
  curves and cocycles for the {M}umford-{M}orita-{M}iller classes.
\newblock {\em Math. Res. Lett.}. 1996, 629--641. MR1418577,  Zbl 0889.14009

\bibitem[KM05]{kotschick2005signatures}
D~Kotschick and S~Morita.
\newblock Signatures of foliated surface bundles and the symplectomorphism
  groups of surfaces.
\newblock {\em Topology}, 44(1):131--149, 2005. MR2104005, Zbl 1163.57304

\bibitem[KM07]{kotschick2004characteristic}
D.~Kotschick and S.~Morita.
\newblock Characteristic classes of foliated surface bundles with
  area-preserving holonomy.
\newblock {\em J. Differential Geom.}, 75(2):273--302, 2007. MR2286823, Zbl 1115.57016

\bibitem[KM09]{kotschick2009gelfand}
D~Kotschick and S~Morita.
\newblock The {G}elfand-{K}alinin-{F}uks class and characteristic classes of
  transversely symplectic foliations.
\newblock {\em arXiv:0910.3414}, 2009.

\bibitem[Kry71]{krygin1971continuation}
AB~Krygin.
\newblock Continuation of diffeomorphisms retaining volume.
\newblock {\em Functional Analysis and Its Applications}, 5(2):147--150, 1971.  MR0368067,  Zbl 0236.57016

\bibitem[McD79]{mcduff1979foliations}
Dusa McDuff.
\newblock Foliations and monoids of embeddings.
\newblock In {\em Geometric topology ({P}roc. {G}eorgia {T}opology {C}onf.,
  {A}thens, {G}a., 1977)}, pages 429--444. Academic Press, New York-London,
  1979.  MR0537744, Zbl 0473.57016

\bibitem[McD81]{mcduff1981groups}
Dusa McDuff.
\newblock On groups of volume-preserving diffeomorphisms and foliations with
  transverse volume form.
\newblock {\em Proceedings of the London Mathematical Society}, 3(2):295--320,
  1981.  MR0628279, Zbl 0411.57028

\bibitem[McD82]{MR707329}
Dusa McDuff.
\newblock Local homology of groups of volume preserving diffeomorphisms. {I}.
\newblock {\em Ann. Sci. \'Ecole Norm. Sup. (4)}, 15(4):609--648 (1983), 1982. MR707329, Zbl 0577.58005

\bibitem[McD83a]{mcduff1983local}
Dusa McDuff.
\newblock Local homology of groups of volume-preserving diffeomorphisms. iii.
\newblock In {\em Annales scientifiques de l'{\'E}cole Normale Sup{\'e}rieure},
  volume~16, pages 529--540. Soci{\'e}t{\'e} math{\'e}matique de France, 1983. MR0740589, Zbl 0619.58008

\bibitem[McD83b]{MR678355}
Dusa McDuff.
\newblock Some canonical cohomology classes on groups of volume preserving
  diffeomorphisms.
\newblock {\em Trans. Amer. Math. Soc.}, 275(1):345--356, 1983. MR678355, Zbl 0522.57029

\bibitem[McD87]{mcduff1987applications}
Dusa McDuff.
\newblock Applications of convex integration to symplectic and contact
  geometry.
\newblock In {\em Annales de l'institut Fourier}, volume~37, pages 107--133, 
  1987. MR0894563, Zbl 0572.58010

\bibitem[Mil83]{milnor1983homology}
John Milnor.
\newblock On the homology of {L}ie groups made discrete.
\newblock {\em Commentarii Mathematici Helvetici}, 58(1):72--85, 1983. MR0699007, Zbl 0528.20033

\bibitem[Mor87]{morita1987characteristic}
Shigeyuki Morita.
\newblock Characteristic classes of surface bundles.
\newblock {\em Inventiones mathematicae}, 90(3):551--577, 1987. MR0914849 ,Zbl 0608.57020

\bibitem[Mos65]{MR0182927}
J{\"u}rgen Moser.
\newblock On the volume elements on a manifold.
\newblock {\em Trans. Amer. Math. Soc.}, 120:286--294, 1965. MR0182927,  Zbl 0141.19407

\bibitem[MT01]{MR1856399}
Ib~Madsen and Ulrike Tillmann.
\newblock The stable mapping class group and {$Q(\bC P^\infty_+)$}.
\newblock {\em Invent. Math.}, 145(3):509--544, 2001. MR1856399, Zbl 1050.55007

\bibitem[MW07]{madsen2007stable}
Ib~Madsen and Michael Weiss.
\newblock The stable moduli space of {R}iemann surfaces: Mumford's conjecture.
\newblock {\em Annals of mathematics}, pages 843--941, 2007. MR2335797 ,Zbl 1156.14021

\bibitem[Nar17a]{nariman2014homologicalstability}
Sam Nariman.
\newblock Homological stability and stable moduli of flat manifold bundles.
\newblock {\em Advances in Mathematics}, 320:1227--1268, 2017.  MR3709135, Zbl 06794793

\bibitem[Nar17b]{nariman2015stable}
Sam Nariman.
\newblock Stable homology of surface diffeomorphism groups made discrete.
\newblock {\em Geom. Topol.}, 21(5):3047--3092, 2017. MR3687114,  Zbl 06794793

\bibitem[Pow78]{powell1978two}
Jerome Powell.
\newblock Two theorems on the mapping class group of a surface.
\newblock {\em Proceedings of the American Mathematical Society},
  68(3):347--350, 1978. MR0494115, Zbl 0391.57009

\bibitem[RW13]{MR3044208}
Oscar Randal-Williams.
\newblock The space of immersed surfaces in a manifold.
\newblock {\em Math. Proc. Cambridge Philos. Soc.}, 154(3):419--438, 2013. MR3044208, Zbl 1276.57035

\bibitem[RW16]{randal2009resolutions}
Oscar Randal-Williams.
\newblock Resolutions of moduli spaces and homological stability.
\newblock {\em Journal of the European Mathematical Society}, 18:1--81, 2016. MR3438379, Zbl 1366.55011

\bibitem[Tho57]{MR0089408}
R.~Thom.
\newblock L'homologie des espaces fonctionnels.
\newblock In {\em Colloque de topologie alg\'ebrique, {L}ouvain, 1956}, pages
  29--39. Georges Thone, Li\`ege; Masson \& Cie, Paris, 1957. MR0089408, Zbl 0077.36301

\end{thebibliography}
\end{document}